\documentclass[12pt]{amsart}

\usepackage[hypertexnames=false, colorlinks, citecolor=red, linkcolor=red, pagebackref]{hyperref}

\usepackage{graphicx}

\usepackage{psfrag}

\usepackage{amsmath} 

\usepackage{amsfonts}

\usepackage{amssymb}

\usepackage{amsthm} 

\newtheorem{theorem}{Theorem}

\newtheorem{conjecture}[theorem]{Conjecture}

\newtheorem{corollary}[theorem]{Corollary}

\newtheorem{lemma}[theorem]{Lemma}

\newtheorem{proposition}[theorem]{Proposition}

\newtheorem{remark}[theorem]{Remark}

\newtheorem*{thma}{Theorem A}

\renewenvironment{proof}[1][Proof]{\textbf{#1.} }{\ \rule{0.5em}{0.5em}}

\newcommand{\RR}{{\mathbb R}}

\newcommand{\CC}{{\mathbb C}}

\newcommand{\NN}{{\mathbb N}}

\newcommand{\ZZ}{{\mathbb Z}}

\newcommand{\MM}{{\mathcal M}_+}

\newcommand{\GG}{{\mathcal G}}

\newcommand{\II}{{\mathcal I}}

\newcommand{\SSS}{{S}}

\newcommand{\PPP}{{P}}

\newcommand{\HHH}{{\mathcal H}}

\newcommand{\GGc}{\check{\mathcal G}}

\newcommand{\EE}{{\mathcal E}}

\newcommand{\MMM}{{\mathcal M}}

\newcommand{\EEE}{{\tilde{\mathcal E}}}

\newcommand{\Cpc}{\mbox{Cap}}

\newcommand{\cpc}{\mbox{Cap}}

\newcommand{\oT}{{\overline{T}}}

\newcommand{\oS}{{\overline{S}}}

\newcommand{\pp}{{p^{\prime}}}

\newcommand{\onepp}{{1-p^{\prime}}}

\newcommand{\orho}{{\overline{\rho}}}

\newcommand{\supp}{\mbox{supp}}

\begin{document}

\title{Potential theory on trees, graphs and Ahlfors-regular metric spaces}

\author[N. Arcozzi]{Nicola Arcozzi}
\address{Dipartimento di Matematica, Universita di Bologna, 40127 Bologna, ITALY}
\email{arcozzi@dm.unibo.it }

\author[R. Rochberg]{Richard Rochberg}
\address{Department of Mathematics, Washington University, St. Louis, MO 63130, U.S.A}
\email{rr@math.wustl.edu}

\author[E. T. Sawyer]{Eric T. Sawyer}
\address{Department of Mathematics \& Statistics, McMaster University; Hamilton,
Ontairo, L8S 4K1, CANADA}
\email{sawyer@mcmaster.ca}

\author[B. D. Wick]{Brett D. Wick}
\address{School of Mathematics, Georgia Institute of Technology, 686 Cherry Street,
Atlanta, GA USA 30332--0160}
\email{wick@math.gatech.edu }

\thanks{The first author's work partially supported by the COFIN project Analisi
Armonica, funded by the Italian Minister for Research}
\thanks{The second author's work supported by the National Science Foundation under
Grant No. 0700238}
\thanks{The third author's work supported by the National Science and Engineering
Council of Canada.}
\thanks{The fourth author's work supported by the National Science Foundation under
Grant No. 1001098  and 0955432}

\subjclass[2010]{XXXX}

\date{\today}
\keywords{Ahlfors-regular metric space, potential theory, capacity}

\begin{abstract}
We investigate connections between potential theories on a Ahlfors-regular metric space $X$, on a graph $G$ associated with $X$, and on the tree $T$ obtained by removing the ``horizontal edges'' in $G$. Applications to the calculation of set capacity are given.
\end{abstract}

\maketitle

\tableofcontents

\section{Introduction}\label{intro}

One of the cornerstones of potential theory is the notion of set capacity, which plays 
there a r\^ole analogous to that played by Lebesgue measure, or Haar measure, in the theory
of $L^p$ spaces and harmonic analysis, to oversimplify a bit. The notion of capacity has
its origins in physics, where it measures the maximum amount of (positive, say) 
electric charge which can be 
carried by a conductor while keeping the potential generated by that charge below a fixed 
threshold. The notion of capacity has been extended to nonlinear potentials, to various metric
space settings, to the theory of stochastic processes and more. Since a long time, capacity had
its peculiar r\^ole in the theory of conformal mappings and it is foundational material for the 
theory of quasi-conformal mappings. 
There is an extensive literature
on these topics and we merely refer the reader to a tiny sample of it \cite{Ts} \cite{AH} 
\cite{Mazya} \cite{Doob} \cite{He}.

A basic fact about set capacity is that it is subadditive, but not additive; to
the point that there are sets of positive, finite capacity with different Hausdorff dimensions.
Capacity, that is, is much different from, although linked with, Hausdorff measure.  
This fact makes it computing, or just estimating, set capacity a craft on its own. 
The capacity  of a set, in fact, measures together its geometric size and the way its pieces are 
distributed in the surrounding space. The bed of nails of a fakir is a well known example: 
cleverly arranged nails sum up to little area, but they might be distributed so as to have sufficient capacity  
not to break the membrane which is laid on them.

There is one context, however, where an elementary algorithm to compute capacity exists,
and that is the context of trees. The capacity of subsets of the tree's boundary, in fact,  
reduces to the calculation of continued fractions of generalized type. This fact seems to have been independently 
observed by all researchers who, for different reasons, developed some potential theory on trees 
\cite{LyPe}\cite{So}\cite{ARS3}.  It is interesting then to know if these simple calculations on trees can be of some 
help in estimating capacities of  sets in, say, Euclidean space.

In this article we show that this is in fact possible in the rather general context of Ahlfors regular metric spaces, 
for $p$-capacities associated with suitable Bessel-type kernels. To each such space $X$ we associate a tree $T$: 
the capacity of a closed subset of $X$ turns out to be estimated from above and below by the capacity of a corresponding 
closed subset of $T$'s boundary. The result seems new even for Bessel capacity in Euclidean space $\RR^n$, $n\ge2$. 
For linear capacities and $n=1$ it was proved by Benjamini and Peres in \cite{BP}.  For $n=1$ and $p\ne2$ it  
follows rather directly from the results in \cite{VW}. The boundary $\partial T$ of the tree $T$ is 
a totally disconnected set with respect to a natural metric. The problem of estimating set capacities is reduced then 
to an analogous problem for subsets of a generalized Cantor set.

The interest of the result, we believe, goes beyond the problem of estimating set capacities \it per se\rm. 
Many problems in potential theory have an essentially combinatorial nature, which is best seen when they are 
translated in the language of trees. 

\smallskip

Our interest in such questions developed while the first three authors were working with 
Carleson measures for Dirichlet-type spaces of holomorphic functions on the unit disc. 
They can can be characterized by means of
a condition involving logarithmic capacity \cite{Ste} or by a discrete testing-type condition making use
of the tree structure of the disc's Whitney decomposition \cite{ARS1}. A \it direct \rm 
proof that testing conditions and capacitary conditions are equivalent in the tree context is in \cite{ARS3}. 
This is evidence that a strict interplay exists in potential theory between some metric spaces and the trees 
obtained by discretizing them.

Observations of a similar flavor have been made before. In \cite{BP}, Benjamini and Peres proved that 
the recurrence/transience dichotomy for some random walks on trees can be decided in terms of logarithmic capacity in the 
Euclidean plane where the trees have been suitably imbedded. In \cite{VW}, Verbitsky and Wheeden proved in some generality 
that the study of nonlinear potentials can be reduced to ``dyadic'' 
potentials. In retrospect, it might be said that the equivalence between dyadic and continuous potentials was implicit in
Wolff's proofs in \cite{HW}, and this article might be seen as a long commentary on Wolff's work. 

The linear, one-dimensional, Euclidean version of the results that are here discussed in greater generality
is used in \cite{ARSW} to prove a Nehari-type theorem for bilinear forms on the holomorphic Dirichlet space.
In \cite{A} it is used to find asymptotic estimates of some condenser capacities in the complex plane. The latter
can be extended to higher dimensions using the results of the present article \cite{AB}.

There are three main tools which we use to build the bridge between trees and Ahlfors-regular spaces. 
The first is Christ's dyadic decomposition of a homogeneous metric space \cite{Christ}; or, rather, the 
easy part of it. We refine it, in the more specialized Ahlfors-regular case, to make room for the singular 
measures which are necessary in potential theory. Christ's construction can be thought of as a  \it graph \rm $G$ 
representing the geometry of $X$ at different scales. In particular, we identify $X$ with the bi-Lipschitz 
image of the graph's boundary, a result which might have independent interest. The tree $T$ used by Christ 
in his study of singular integrals is a \it spanning tree \rm for the graph $G$. A construction closely
related to ours was carried out in \cite{BoPa} \S2. We thank the anonymous referee for directing our 
attention on this and other references. 

The second tool, which is nedded if $X$ is not homeomorphic to a subset of the real line, 
is a technical lemma showing that measures can be moved back and forth from the 
metric space $X$ to the boundary of its associated tree. We must define the (non-canonical) pullback of a measure and 
this poses a nontrivial measurability problem.  The third tool, useful to deal with the nonlinear case, is a deep 
inequality by Muckenhoupt and Wheeden \cite{MW}, later independently proved by Wolff \cite{HW} with a wholly 
different, almost combinatorial argument. The Muckenhoupt-Wheeden-Wolff inequality, in fact, moves potentials 
from the space $X$ to the graph $G$, where it is easy to see that the ``vertical edges'', those which define the 
tree structure, carry all the relevant quantitative information. The inequality of Muckenhoupt, Wheeden and Wolff 
extends the well known equivalence, for subsets of the real line, of planar logarithmic capacity and one-dimensional 
$1/2$-Bessel capacity. In the general case, this equivalence can be stated in terms of Wolff 
potentials.

In order to better focus on the main ideas, in this article we only consider the case of \it bounded, \rm complete 
Ahlfors-regular spaces.  The theory could be extended to the unbounded case by choosing Bessel-type potentials 
with exponential decay at infinity, in order to control the tail estimates. A good source for this kind of extension 
is \cite{KV}.

We finish with a comment about the bibliography. We have made references to articles and books where the results 
we needed or we made reference to could be found. We did not look for the primary source of the results we have 
quoted and we have probably omitted some important references. This is especially true for the section concerning 
potential theory on trees. Many of the results which we prove in Section \ref{pttII} can be found in the literature, 
often independently proved by several authors at different times, including the authors of the present paper, with 
different degrees of generality; or they might be seen as particular cases of general results in one of the several 
axiomatic versions of Potential Theory. We have chosen to prove ourselves what we needed, trying to keep the 
exposition as self-contained as possible. There are two exceptions: the proof of Christ's Decomposition Theorem, for which
we refer to \cite{Christ}, and the basic properties of Axiomatic Potential Theory,  which we took from \cite{AH}. 
The main point of the paper, in fact, is not developing some new Potential Theory 
on Trees, but rather shedding light on the connection between discrete and non-discrete Potential Theory.
Some of our results on trees might nonetheless be new, especially those relating capacity and Carleson 
measures.

We thank the careful referee for pointing out a number of typos in the original draft of the paper and for 
his/her numerous suggestions on how to improve the presentation.

\subsection{Main Results and Outline of the Contents}
Let $(X, m,\rho)$ be a \it Ahlfors-regular metric measure space\rm. \it  By this we mean that 
$(X,\rho)$ is a complete metric space, $ m\ge0$ is a Borel measure on $X$, and there exist 
constants $0<c_1<c_2$, $Q>0$, such that, for all $r\ge0$ and $x\in X$:

\begin{equation}
 \label{ahlfors}
c_1r^Q\le  m(B(x,r))\le c_2 r^Q.
\end{equation}
Here, $B(x,r)=\{y\in X:\ \rho(x,y)<\rho\}$ is the metric ball of radius $r$, centered at $x$. 

In this article we are interested in local 
properties and we frequently will assume that $\mbox{diam}(X)<+\infty$. \rm 
In this case, (\ref{ahlfors}) is only required to hold for $0\le r\le \mbox{diam}(X)$. This special assumption
could be removed by assuming that the potential kernel $K$ below has exponential decay,
analogous to that of the Bessel kernels in $\RR^n$.

Given $0<s<1$, define the kernel $K:X\times X\to[0,\infty]$,
\begin{equation}\label{nucleo}
K(x,y)=[ m(B(x,\rho(x,y)))+ m(B(y,\rho(x,y)))]^{-s}\approx  m(x,\rho(x,y))^{-s}.
\end{equation}
When $X$ is an Ahlfors-regular, bounded domain in $\RR^n$, the kernel $K$ is a Riesz-Bessel
kernel:
$$
K(x,y)=\frac{1}{|x-y|^{sn}}.
$$
Given $1<p,p^\prime<\infty$, $\frac{1}{p}+\frac{1}{p^\prime}=1$, and a Borel measure $\omega\ge0$ on $X$,  consider the \it $p$-energy of $\omega$ \rm associated with the kernel $K$:

$$
\EE_X(\omega):=\int_X[K\omega(x)]^\pp d m(x),
$$
where $K\omega(x)=\int_XK(x,y)d\omega(y)$. The \it $p$-capacity \rm relative to the kernel $K$ 
associates to $E$, a compact subset of $X$, the nonnegative number

\begin{equation}\label{capacity}
\cpc_X(E)=\sup_{\supp(\omega)\subseteq E}\left[\frac{\omega(E)^p}{\EE_X(\omega)^{p-1}}\right].
\end{equation}

See \cite{AH} \cite{He} \cite{HaK} \cite{HeK} \cite{KV} for various approaches to nonlinear 
capacities in measure metric spaces. The kind of potential theory which is discussed in this 
paper is only interesting when the parameter of smoothness $s$ satisfies 
$\frac{1}{p^\prime}\le s<1$. When $0<s<\frac{1}{p^\prime}$, 
singletons have positive capacity.

We will show that capacities in Ahlfors-regular metric spaces can be estimated by similar 
capacities defined on a totally disconnected metric space.

\begin{theorem}\label{maine} To $(X, m,\rho)$ we can associate a tree $T$, 
a metric $\rho_T$ on $\partial T$ with respect to which $\partial T$, 
the boundary of $T$, is Ahlfors $Q$-regular, and a Lipschitz map
$$
\Lambda:\partial T\to X
$$
in such a way that, if $E$ is compact in $X$, then
\begin{equation}\label{mickey}
\cpc_X(E)\approx\cpc_{\partial T}(\Lambda^{-1}(E)).
\end{equation}

In the other direction we have that, if $F$ is compact in $\partial T$, then

\begin{equation}\label{mouse}
\cpc_X(\Lambda(F))\approx\cpc_{\partial T}(F).
\end{equation}
\end{theorem}

Here, $\cpc_{\partial T}$ is the capacity on the measure metric space $(\partial T,m_T,\rho_T)$, defined by the  same 
parameters used in defining the capacity on $X$. More precisely: $Q$ is the Hausdorff dimension of $\partial_T$ 
as well as that of $X$; $ m_T$ is the $Q$-Hausdorff measure on $\partial T$;  
$p$, the integrability exponent; $s$ is the smoothness parameter the parameter; 
the kernel used to define $\cpc_{\partial T}$ is
$$
K_T(x,y)=[ m_T(x,\rho_T(x,y))+ m_T(y,\rho_T(x,y))]^{-s}
$$
One of the main difficulties here is that the map $\Lambda$ is not one-to-one, therefore sets having positive capacity in 
$X$ have multiple preimages, possibly far away, in $\partial T$. This difficulty does not arise when $X$ is homeomorphic
 to a subset of the real line, since in this case the set of the points having multiple preimages is countable.

Let us mention an immediate consequence of Theorem \ref{maine} in the case of Euclidean space
$\RR^n$. The tree of the Euclidean dyadic cubes has a huge group of automorphisms, 
which are in turn isometries of the corresponding measure-metric structure 
$(T,\rho_T, m_T)$. Such automorphisms are generated by the operation of freely shuffling the $2^n$ cubes which are 
immediately below any given cube in the tree. \it The capacity of a set is essentially invariant under such
automorphism, \rm with multiplicative constants which are independent of the set and of the automorphism.
This invariance property is obviously much more general than invariance under isometries of the Euclidean space,
although not as sharp.
Invariance under tree automorphisms is not obvious, especially in several variables: in higher dimensions portions 
of the set having positive capacity might be doubled and moved apart one from the other, operations which generally 
increase capacity.

\smallskip

Here is an outline of the article. In Section \ref{sect} we construct the 
graph $G$ and the tree $T$ associated with $X$, and the surjection 
$\Lambda:\partial T\to X$. Among the properties of $\Lambda$ in the section, 
we show that the natural push-forward of measures $ m\mapsto\Lambda_* m$ has 
a (non-natural) right inverse. This allows us to move measures back and forth 
from $\partial T$ to $X$. 

In Section \ref{sectmw} we prove a version of the Muckenhoupt-Wheeden-Wolff inequality 
for potentials on graphs. The part of the inequality we need is that independently 
proved by T. Wolff. In Section \ref{sectcap} we show that measures on $\partial T$
and $X$, if they are correspondent under $\Lambda_*$ or its inverse map, have comparable 
energy. As a consequence, we prove Theorem \ref{maine}.  

Section \ref{pttII} presents the basic facts of potential theory on trees. First, we show that the 
Muckenhoupt-Wheeden-Wolff inequality establishes a correspondence between ``Bessel'' potential theory on $X$, 
hence on $\partial T$, and  ``logarithmic'' potential theory on $T\cup\partial T$. The latter is the sort 
of potential theory which has been 
independently developed by several authors over the past twenty years, an account of which is given in the rest of the 
section.  Potential theory on trees is simpler than its counterparts on general Ahlfors-regular spaces for several reasons. 
The main such reason is that capacities can be computed explicitly by means of recursive formulas, which are deduced in the 
present article in a rather general form. 
Moreover, in the context of trees, sets' boundaries are often trivial and scaling arguments are elementary 
and natural. As a consequence, the capacitary potential of a set $E$ is a simple object, whose geometry is linked to that of 
the set $E$ in a very transparent way.

Some of the material presented here has already appeared in the literature, while some of it is presented here for the first 
time.  We present, in particular, a deduction of the trace inequality for the potential $K$ from a much simpler dyadic 
``Carleson measure inequality'' on trees.  Then, we give a rather \it direct \rm proof that a testing condition known to 
be equivalent to the trace inequality implies a capacitary condition, also known to be equivalent to the trace inequality.
 This answers a question Maz'ya asked some of us some years ago.

%In Section \ref{last} we apply 

%some results from Section \ref{pttII} to the study of potential theory on Ahlfors spaces. 

For ease of the reader, in the Appendix 
we have translated in tree terms some basic results of Nonlinear Potential Theory, as they are presented in 
\cite{AH} (Sect. 2.3-2.4). 

\medskip

Our hope is that the simple tree model will be useful to researchers working in or using the results of potential theory. 
In this paper, we do not offer new applications of the equivalence between ``classical'' and tree capacities, except for a 
new proof of the trace inequalities in the Ahlfors-regular case. We hope to return on applications in other papers and that 
other people will find useful the tool we have here developed.

\subsection{Some Examples of Ahlfors-Regular Metric Spaces.} We end the introduction by mentioning a  few, 
often not independent, examples of the spaces to which the theory developed in the present paper applies.   
More material, at a much deeper level, can be found, for example, in \cite{He}.

\begin{enumerate}

\item Euclidean $\RR^n$ and the $n$-dimensional sphere $\Sigma_n$ are $n$-dimensional Ahlfors-regular spaces.

\item Both the closed and open balls in Euclidean space $\RR^n$ are $n$-regular Ahlfors spaces.

\item The ternary Cantor set $C$ with the metric inherited from the real line is Ahlfors-regular with 
$Q=\frac{\log2}{\log3}$.

\item If $(X,\rho, m)$ is Ahlfors $Q$-regular, the measure $ m$ can be replaced by the $Q$-Hausdorff measure 
$H^Q_\rho$ in $(X,\rho)$. The metric measure structure, that is, can be reduced to the metric structure alone.  

\item The ``snowflake metric'' $\rho(x,y)=|x-y|^{1/2}$ makes the real line into an Ahlfors $2$-regular space.

\item Carnot groups having homogeneous dimension $Q$ with their Carnot metrics are Ahlfors $Q$-regular. 
See \cite{BLU} for a comprehensive introduction to the topic.

\item An $n$-dimensional, complete Riemannian manifold with nonnegative, bounded curvature is Ahlfors 
$n$-regular.

\item The unit sphere in $\CC^n$ endowed with the Koranyi metric $\rho(z,w)=|1-z\cdot\overline{w}|^{1/2}$ is an 
Ahlfors-regular space having dimension $2n+1$.

\item If $(X,\rho, m)$ is Ahlfors $Q$-regular and $Q_\alpha$ is one of the dyadic boxes in $X$'s dyadic 
decomposition (see Christ's 
Theorem in the next section), then $\overline{Q}_\alpha$, the closure of $Q_\alpha$ in $X$, is $Q$-regular.  
See Theorem \ref{subdyadic}.

\item The boundary of any homogeneous tree (except $\ZZ$) with respect to the Gromov distance is a regular Ahlfors space.
As before, the measure considered here is the $Q$-Haudorff measure, where $Q$ is the Hausdorff dimension of $\partial T$.

\end{enumerate}

\paragraph{Notation.} \it If $A(P_1,\dots,P_n)$ and $B(P_1,\dots,P_n)$ are two positive, or positively infinite, quantities 
depending on the objects $P_1,\dots,P_n$, we write $A\approx B$ if there are constants $0<C_1<C_2$, 
independent of $P_1,\dots,P_n$, s.t $C_1 A\le B\le C_2 A$. We write $A\lesssim B$ if there is a constant $C>0$ such 
that $A\le CB$. We will denote by $c$ a positive constant which might change value within the same expression or 
calculation.

\rm

\section{A Metric Space, a Graph, and a Tree}
\label{sect} 

\it

In this section we consider the discretization of the Ahlfors-regular space $X$ at different scales by means of metric
``dyadic boxes`` as proved by
 M. Christ. While we only need the \rm easy \it part of Christ's Theorem, we have to be careful: having to deal with Borel 
measures, we can not discard sets having null $ m$-measure.
The difficult
 part of Christ's is an estimate of the mass concentrated near the boundary of the dyadic boxes.
 Christ's Theorem can be interpreted as the construction of a 
\bf graph \it $G$, having  $X$ as boundary.  The vertices of the graph are dyadic boxes, which generalize to the 
context of Ahlfors-regular metric spaces the dyadic decomposition of Euclidean space: each dyadic box is partitioned 
into a number of smaller dyadic boxes, which we might call its ``children``.  The edges are
either ``horizontal'' or ``vertical''. Horizontal edges join nearly adjacent dyadic boxes having comparable diameter, while
a vertical edge joins a dyadic box with its ``parent`` dyadic box.
  
To better clarify this point we bi-Lipschitz modify the original distance 
$\rho$ to $\overline{\rho}$, a distance which extends to a length-distance on $\overline{G}:=G\cup X$.  We then consider 
the \bf tree structure \it $T$ of the dyadic boxes: $T$ and $G$ have the same vertices, but $T$ only preserves the vertical 
edges. The boundary $\partial T$ of $T$ is ``larger'' than $X$, and can be thought of as a Cantor set. We construct an onto,
Lipschitz map $\Lambda:\overline{T}\to\overline{G}$, mapping boundaries to boundaries and such that 
$\Lambda(\partial T)=X$. The 
main point of this section consists in proving that the push-forward of measures $\Lambda_*\omega:=\omega\circ\Lambda$ 
has a (non-canonical) left inverse with nice properties. The main technical difficulty consists in proving that $\Lambda$ 
maps Borel sets to Borel sets or, rather, a generalization of this fact. Once all this is done, we can move Borel measures 
back and forth through $\Lambda$.

\rm

\subsection{Christ's Theorem Revisited}

Let $(X, m,\rho)$ be an Ahlfors $Q$-regular metric measure space. We denote the balls by $B(x,r)=\{y\in X:\ \rho(y,x)<r\}$. 
\it In this section we do not assume $X$ to be bounded. \rm

In particular, $(X,\rho)$ is a \it homogeneous space \`a la Coifman-Weiss \rm: the measure $ m$ satisfies the doubling 
condition
$$
 m(B(x,2r))\le c_3 m(B(x,r)).
$$

It is well known (see, e.g., Semmes' essay in \cite{Gro}) that we can take $ m$ to be the $Q$-dimensional Hausdorff measure 
for the metric space $(X,\rho)$. Namely, if $(X,m,\rho)$ is Ahlfors $Q$-regular and ${\mathcal H}^Q_\rho$ is the 
$Q$-dimensional Hausdorff measure for the distance $\rho$, then $(X, {\mathcal H}^Q_\rho, \rho)$
is Ahlfors $Q$-regular and there is $A>0$ such that
$$
A^{-1} m(E)\le{\mathcal H}^Q_\rho(E)\le m(E)
$$
for all Borel measurable sets in $X$.

We shall also assume, in the main body of the paper, that $X$ is bounded, $diam(X)\le 1$, so that (\ref{ahlfors}) only is 
required to hold when $r\le1$. Under the sole hypothesis that $(X,\rho, m)$ is a homogeneous space, M. Christ \cite{Christ} 
proved that $X$ admits a ``dyadic'' decomposition. We state here \it the easy part \rm of Christ's Theorem. 

\begin{thma}\label{christus}[M. Christ]

There exists a collection $\{Q^k_\alpha,\ \alpha\in I_k,\ k\in\ZZ\}$ of open subsets of $X$ and $\delta>0$, $a_0>0$, $c_4>0$ such that

\begin{enumerate}

 \item[(i)] $ m(X\setminus\cup_{\alpha\in I_k}Q^k_\alpha)=0$ holds for all $k\ge0$; 

\item[(ii)] if $l\ge k$ then, for all $\alpha\in I_k$ and $\beta\in I_{l}$, 
either $Q^l_\beta\subseteq Q^k_\alpha$, or $Q^l_\beta \cap Q^k_\alpha=\emptyset$;

\item[(iii)] for all $(l,\beta)$ and $l>k$, there is a unique $\alpha$ in $I_k$ such that $Q^l_\beta\subseteq Q^k_\alpha$;

\item[(iv)] $diam(Q^k_\alpha)\le c_4 \delta^k$;

\item[(v)] for all $(k,\alpha)$, $B(z_\alpha^k,\delta^k)\subseteq Q^k_\alpha$, for distinguished points $z^k_\alpha\in Q^k_\alpha$.

\end{enumerate}

\end{thma}

If $X$ is bounded, after rescaling the distance 
we can assume that $k\in\NN$ and that $I_0$ contains a unique element $o$: $Q^0_o=X$.
By analogy with the case of Euclidean space, we will call the sets $Q_{\alpha}^k$ or their 
closures \it dyadic sets \rm or \it qubes. \rm

We take the sets $\alpha\equiv(k,\alpha)\leftrightarrow Q^k_\alpha$ themselves as points of a new space $T$. When we do not want to stress the level of $\alpha$, we simply write $Q_\alpha$ instead of $Q_\alpha^k$. 

For $\alpha\in I_k$, we set $k=d(\alpha)$ to be the level of $\alpha$. The set $T$ is given a \it tree structure\rm: there is an edge of the tree between $\alpha$ and $\beta$ if $d(\beta)=d(\alpha)+1$ and $\beta\subseteq\alpha$, possibly interchanging the 
r\^oles of $\alpha$ and $\beta$. The tree $T$ has a natural, edge-counting distance $d$, which is realized by geodesics. We introduce the partial order: $\alpha\le\beta$ in $T$ if $\alpha\in[o,\beta]$, the geodesic joining $o$ and $\beta$.

We also introduce on $T$ a further \it graph structure\rm: two distinct points $\alpha$ and $\beta$ in $T$ are connected by an edge 
of the graph $G$ if they are already connected by an edge of $T$, or if $d(\alpha)=d(\beta)=k$ and there are points 
$x\in\alpha$ and $y\in\beta$ such that $\rho(x,y)\le\delta^k$. In this case, we write 
$\alpha\genfrac{}{}{0pt}{}{\sim}{G}\beta$. 

The natural edge-counting distance in $G$ is denoted by $d_G$. It is realized by geodesics, but, contrary to the tree case, 
there might be several geodesics joining two points.  

We identify $T=G$ as sets of vertices, and use different names when 
dealing with different edge structures. It would be more elegant to use two names for the 
different edge sets and one third name for the common vertex set. 
We have here sacrified elegance to economy, more in line with the dialect of ``dyadic'' analysts than
with the language of graph theorists.

In Graph Theory it is said that $T$ is a \it spanning tree \rm for the graph $G$.

Since we are interested in possibly singular measures on $X$, we have to refine Christ's construction a 
little in order to 
have regions whose union gives us back the whole space $X$. This is done in the following lemmas.

\begin{lemma}
 \label{euno} 
For each $k$, we have that $\cup_{\alpha\in I_k}\overline{Q^k_\alpha}=X$.
\end{lemma}

\proof
By contradiction, suppose that $x$ lies in $X\setminus\cup_{\alpha\in I_k}\overline{Q^k_\alpha}$. Then, any ball  
$B(x,\epsilon)$ intersects $Q^k_\alpha$ for infinitely many $\alpha$ in $I_k$. In fact, since $ m(B(x,\epsilon))>0$, 
by (i) in Theorem A and the assumption that open sets have, by Ahlfors-regularity, positive measure,  
there exists $\alpha_1$ in $I_k$ such that $B(x,\epsilon)\cap Q^k_{\alpha_1}\ne\emptyset$. 
There must be some $0<\epsilon_2<\epsilon=\epsilon_1$ such that $B(x,\epsilon_2)\cap Q^k_{\alpha_1}=\emptyset$, 
otherwise $x\in\overline{Q^k_{\alpha_1}}$, a contradiction. As before, there must be then $\alpha_2\ne\alpha_1,\alpha$ 
such that $B(x,\epsilon_2)\cap Q^k_{\alpha_2}\ne\emptyset$.

Iterating this procedure, we find a sequence $Q^k_{\alpha_n}$ of distinct
regions at generation $k$. They have to be disjoint by (ii) in Christ's Theorem. 
Also, they are contained in the ball $B(x,\epsilon+2c_4\delta^k)$, by (iv), and each of 
them has volume at least $c_1\delta^{Qk}$, by Ahlfors-regularity (\ref{ahlfors}). Hence,
$$
+\infty=\sum_n m(Q^k_{\alpha_n})= m(\cup_nQ^k_{\alpha_n})\le m(B(x,\epsilon+2c_4\delta^k))<+\infty,
$$
a contradiction.
\endproof

Lemma \ref{euno} is trivial if $\text{diam}(X)<\infty$, having in this case $I_k$ finitely many elements only.
From the lemma, the fact that the regions $Q^k_\alpha$ are open, and the regularity of the measure $m$, we deduce 
immediately the following.
\begin{corollary}
\label{eleffe}
Let $F^k=\cup_{\alpha\in I_k}\partial Q^k_\alpha$ and $F=\cup_k F^k$. Then, each $F^k$ is closed,
$F^k\subseteq F^{k+1}$ and $ m(F)=0$.
\end{corollary}
\begin{proof}
We first show that $F^k$ is closed. 
Let $x_n\in \partial Q_{\alpha_n}^k$, $x_n\to x$ in $X$. By Lemma \ref{euno}, if $x\notin F^k$, then $x\in Q_{\alpha}^k$
for some $\alpha\in I_k$, but $Q_\alpha^k$ is open and the the $Q_\beta^k$'s are disjoint, hence $x_n\notin F^k$ for large 
$n$:
a contradiction.

The inclusion $F^{k}\subseteq F^{k+1}$ easily follows from Lemma \ref{euno} and (iii) in Christ's Theorem. 

Assertion (i) in Theorem \ref{christus} implies that $m(F^k)=0$, and $m(F)=0$ by regularity of the measure $m$.
\end{proof}

\begin{lemma}
\label{edue}
\hfill
\begin{itemize}
\item[(i)] There exists $c_5$ independent of $k\ge0$ and of $x\in X$ such that 
$\sharp\{\alpha\in I_k:\ x\in \overline{Q^k_\alpha}\}\le c_5$;\\ 
\item[(ii)] There exists $c_6$ independent of $\alpha$ in $G$ such that  
$\sharp\{\beta\in G:\ \beta\genfrac{}{}{0pt}{}{\sim}{G}\alpha\}\le c_6$: the graph $G$ has 
{\bf bounded degree}.
\item[(iii)] In particular, the number of $\beta\in I_{k+1}$ such that $Q_\beta^{k+1}$ is contained in a fixed $Q_\alpha^k$
is bounded by a constant $c$ which only depends by the structural constants of the metric space $(X,\rho)$.
\end{itemize}
\end{lemma}
\begin{proof}
It is clear that (i) is a consequence of (ii). In fact, all $\alpha\in I_k$ for which $x$ lies in $\overline{Q_\alpha^k}$  
are $G$-related.

Fix $x\in Q^k_\alpha$ and suppose that $\beta\genfrac{}{}{0pt}{}{\sim}{G}\alpha$. Consider first the case 
$\beta\in I_k$.
There is then a constant $c$ such that, 
if $\rho(Q^k_\alpha, Q^k_\beta)\le\delta^Q$, then $Q^k_\beta\subseteq B(x,c\delta^k)$. Now, using Ahlfors $Q$-regularity
and items (iv), (v) of Christ's Theorem in the first inequality,
\begin{eqnarray*}
 \delta^{Qk}\cdot\sharp\{\beta\in I_k:\ Q^k_\beta\cap B(x,\delta^k)\ne\emptyset\}
&\le&c\sum_{Q^k_\beta\cap B(x,\delta^k)\ne\emptyset} m(Q^k_\beta)\crcr
&=& m(\cup_{Q^k_\beta\cap B(x,\delta^k)\ne\emptyset}Q^k_\beta)\crcr
&\le& m(B(x,c\delta^k))\le c\delta^{Qk},
\end{eqnarray*}
which implies (ii) under the restriction $\beta\in I_k$. The general case easily follows from the special one 
and Ahlfors-regularity. 
\end{proof}

\begin{lemma}\label{localizing}
 If $\alpha\in I_k$ and $h\ge0$, then
$$
\overline{Q}_\alpha=\cup_{\beta\ge\alpha,\ \beta\in I_{k+h}}\overline{Q}_\beta.
$$
\end{lemma}
\begin{proof} 
Clearly, $\overline{Q}_\alpha\supseteq\cup_{\beta\ge\alpha,\ \beta\in I_{k+h}}\overline{Q}_\beta.$  In the other 
direction, since the $Q_\beta$'s are finitely many, it suffices to show that
$\cup_{\beta\ge\alpha,\ \beta\in I_{k+h}}\overline{Q}_\beta$ is dense in $Q_\alpha$. 
If this were not the case, there would be a metric ball $B(z,\epsilon)$ in 
$Q_\alpha\setminus\left[\cup_{\beta\ge\alpha,\ \beta\in I_{k+h}}\overline{Q}_\beta\right]$, hence, by Ahlfors-regularity,
$$
 m(Q_\alpha)>\sum_{\beta\ge\alpha,\ \beta\in I_{k+h}} m(Q_\beta),
$$
which, (i) and (iii) in the Theorem of Christ, contradicts Corollary \ref{eleffe}.
\end{proof}

Complete Ahlfors-regular spaces are not especially exotic among metric spaces.
\begin{corollary}
\label{heineborel}
$(X,\rho)$ is locally compact.
\end{corollary}
\begin{proof}
Mimic the proof of the Heine-Borel Theorem in \cite{Ru}, p.38-40 using the dyadic decomposition of Christ instead 
of the usual dyadic decomposition of $\RR^n$, and the metric completeness of $X$ instead of the completeness of $\RR$.
\end{proof}

\begin{corollary}
\label{tonite}
$(X,\rho)$ is separable.
\end{corollary}
\begin{proof}
It suffices to show that each $\overline{Q}_\alpha$ is separable.  
By Christ Theorem, for each integer $n>0$, there are finitely many points 
$\{z^n_j:\ j=1,\dots,k_n\}$ in $\overline{Q}_\alpha$ such that any point 
$x$ in $\overline{Q}_\alpha$ lies at distance at most $2^{-n}$ from some $z^n_j$. 
The countable set $\{z^n_j:\ j=1,\dots,k_n,\ n\ge1\}$ is dense in $\overline{Q}_\alpha$.
\end{proof}

\begin{theorem}
\label{subdyadic}
Each set $X_\alpha=\overline{Q}_{\alpha}$ is Ahlfors $Q$-regular. Moreover, the dyadic sets in Christ's 
decomposition of $X_\alpha$ might be taken to be the sets $Q_\beta$ where $\beta\ge\alpha$. 
\end{theorem}
\begin{proof}Assume without loss of generality that $\alpha\in I_0$ and let $x$ be a point of $\overline{Q}_\alpha$. Let $r\in(0,1)$, $\delta^{k+1}<r\le\delta^k$, and let $\beta\in I_k$, $\beta\ge\alpha$, be such that $x\in\overline{Q}_\beta$ (at least one such $\beta$ exists by Lemma \ref{localizing}). By Christ's Theorem, there is constant $C$ such that $Q_\alpha\cap B(x,C\delta^k)\supseteq Q_\beta$, hence,
$$
 m(Q_\alpha\cap B(x,C\delta^k))\ge m(Q_\beta)\ge C^\prime\delta^{Qk}.
$$
Ahlfors-regularity of $\overline{Q}_\alpha$ immediately follows.
\end{proof}

\subsection{The space \texorpdfstring{$X$}{X} as the Boundary of \texorpdfstring{$G$}{G}}

We now define a new distance $\orho$ on $G\cup X$. Let $\alpha\ne\beta$ be points of $G$ such that 
$d(\alpha)\le d(\beta)\le d(\alpha)+1$ and $\alpha\genfrac{}{}{0pt}{}{\sim}{G}\beta$. 

To the edge $[\alpha,\beta]$ we associate the weight $\ell_G([\alpha,\beta]):=\delta^{d(\alpha)}$. 
The length of a path $\Gamma=(\alpha_0,\dots,\alpha_n)$ in $G$ is defined by adding the length of its edges. 

If needed, we set $\ell_G((\alpha))=0$,  where $(\alpha)$ is the trivial path from $\alpha$ to itself. The distance 
$\orho$ is defined, on  points $\alpha,\beta\in G$, as

\begin{equation}
\label{orho}
\orho(\alpha,\beta)=\inf\{\ell_G(\Gamma):\ \Gamma\ \mbox{a\ path\ having\ endpoints\ }\alpha
\ \mbox{and}\ \beta\}.
\end{equation}

By the triangle inequality,

\begin{equation}\label{rhorho}
\rho(Q_\alpha,Q_\beta)\le c\orho(\alpha,\beta).
\end{equation}

Let $(\overline{G},\orho)$ be the completion of the metric space $(G,\orho)$. By this we mean the usual Cantor
construction: Cauchy sequences $\{\alpha_n\}$ and $\{\beta_n\}$ in $(G,\rho)$ are identified if 
$\lim_{n\to\infty}\rho(\alpha_n,\beta_n)=0$.  
Each point $\alpha\in G$, identified with the equivalence class of the sequences in$G$ taking value
$\alpha$ definitely, is isolated in 
$\overline{G}$. Let $\partial G:=\overline{G}\setminus G$.  
There is a bijection $\flat$ between $\partial G$ and $X$.

The metric space  $\overline{G}$ is the \it Floyd completion \rm of $G$. We will show in Theorem \ref{lengthmetric} 
that the map $\flat$ is bi-Lipschitz and onto. A more general result is stated in Theorem 4.7 of the very informative
\cite{Geo}. It is also shown there that the weights can be chosen in such a way $\flat$ becomes an isometry.
We thank the referee for making us aware of the article \cite{Geo}, that also contains a detailed description of 
much related literature. 
In order to keep the exposition as self-contained as possible, and close to the
viewpoint given by Christ's Theorem, we provide our own proof of Theorem \ref{lengthmetric}.

\begin{lemma}
\label{puntinisullei}
Let $a=\{\alpha_n\}$ be Cauchy sequence  in $(G,\orho)$, which is not definitely constant, and let 
$[a]$ be the equivalence class of $a$ in $\partial G$. Define
$$
\flat([a])=x\in X
$$
if and only if $\lim_{n\to\infty}\rho(x,Q_{\alpha_n})=0$.  Then, $\flat$ is a well defined bijection of $\partial G$ onto $X$.
\end{lemma}
\begin{proof}
The proof will be broken up into several steps.  
\begin{itemize}

\item \it To each Cauchy sequence $\flat$ associates a unique $x\in X$. \rm 

Let $d(\alpha_n)$ be the level of $\alpha_n$ in $T$.  Since the  sequence is not definitely constant, 
it is Cauchy and $G$ is discrete, we must have $\lim_{n\to\infty}d(\alpha_n)=+\infty$. Let $z_{n}$ be a 
distinguished point in $Q_{\alpha_n}$. Let $n,m\ge1$. Using (\ref{rhorho}),
\begin{eqnarray*}
\rho(z_n,z_{n+m})&\le& \mbox{diam}(Q_{\alpha_n})+\rho(Q_{\alpha_n},Q_{\alpha_{n+m}})
+\mbox{diam}(Q_{\alpha_{n+m}})\crcr
&\le&c(\delta^{d(\alpha_n)}+\orho(\alpha_n,\alpha_{n+m})+\delta^{d(\alpha_{n+m})})
\to 0,
\end{eqnarray*}
as $n\to\infty$. Since $(X,\rho)$ is complete and $\{z_n\}$ is Cauchy, it has a well defined limit $x$ and 
$\lim_{n\to\infty}\rho(x,Q_{\alpha_n})\le\lim_{n\to\infty}\rho(x,z_n)=0$. This shows existence of $x$.

Being the sequence $\alpha_n$ non-definitely constant, the $\rho$-diameter of $Q_{\alpha_n}$ tends to
$0$ as $n\to\infty$, which immediately implies the uniqueness of $x$ with the property .
$\lim_{n\to\infty}\rho(x,Q_{\alpha_n})=0$.

\item \it If $[a]=[b]$, $a$ and $b$ define the same point in $X$: $\flat([a])=\flat([b])$. \rm 

Let $a=\{\alpha_n\}$ and $b=\{\beta_n\}$.  Choose distinguished points $z_n\in Q_{\alpha_n}$ and 
$w_n\in Q_{\beta_n}$. Reasoning as above, $\lim_{n\to\infty}\rho(z_n,w_n)=0$, hence the two sequences converge to the same point $x\in X$.

\item \it $\flat:\partial G\to X$ is surjective. \rm 

Given $x$ in $X$, for each $n\ge0$ we can choose by Lemma \ref{euno} a $Q^n_{\alpha_n}$ such that 
$x\in\overline{Q^n_{\alpha_n}}$. It is easily checked that the graph distance between $(n,\alpha_n)$ and 
$(n+1,\alpha_{n+1})$ is at most $2$ and that $\flat([\{(n,\alpha_n)\}])=x$. 

\item \it $\flat$ is injective. \rm
Suppose that $a=\{{\alpha_n}\},\ b=\{{\beta_n}\}$ are Cauchy sequences in $(G,\orho)$, not definitely
constant, and that $\flat([a])=\flat([b])$. This means that there is $x\in X$ is such that 
$\lim_{n\to\infty}\rho(Q_{\alpha_n},x)=\lim_{n\to\infty}\rho(Q_{\beta_n},x)=0$. Since the diameters of 
$Q_{\alpha_n}$ and $Q_{\beta_n}$ tend to zero and, by the triangle property, 
$\rho(Q_{\alpha_n},Q_{\alpha_n})$ tends to zero, $[a]=[b]$.

%If $x\ne y$, $\rho(x,y)\ge\delta^{k}$ for some $k$. Let $a=\{\alpha_n\}$ and $b=\{\beta_n\}$ be such 
%that $x=\flat([a])$ and $y=\flat([b])$. Since $d(\alpha_n),d(\beta_n)\to\infty$, for some $n_0$ and 
%$n\ge n_0$ we have: $n\ge k$ and 
%$\rho(\overline{Q^n_{\alpha_n}},\overline{Q^n_{\beta_n}})\ge \frac{1}{2} \delta^{k}$. 
%Then, $\orho(\alpha_n,\beta_n)\ge \epsilon\delta^k$ for large $n$ ($\epsilon$ fixed), which in turn 
%implies $[a]\ne[b]$.

\end{itemize}
\end{proof}

%% [I don't know if we need this stuff...]

%Let now $x$ be a point of $x$ and let $\Gamma^{k_0}_\zeta=(\alpha_k)_{k\ge k_0}$ 

%be a tree geodesic starting at level $k_0$ such that $x=\cap_k \overline{Q^k_{\alpha_k}}$ 

%($\zeta$ is the point of $\partial T$ corresponding to the geodesic $\Gamma^{k_0}_\zeta$). 

%There might be several geodesics starting at level $k_0$ and ending at $\zeta\in\partial T$, 

%although boundedly many, as we shall see below. The length of $\Gamma^{k_0}_\zeta$ 

%is the sum of its edges: 

%$\ell_G(\Gamma^{k_0}_\zeta)=\sum_{k\ge k_0}\delta^k=\frac{\delta^{k_0}}{1-\delta}$.

% We say that $x$ is an \it endpoint \rm of $\Gamma^{k_0}_\zeta$.

%We will consider paths $\Gamma$ 

%in $G$, having endpoints in $\overline{G}:=G\cup X$.\footnote{A formal definition 

%for such paths}

%Given $\alpha,\beta$ in $\overline{G}$, (\ref{orho}) extends the metric $\orho$ to 

%$\overline{G}$.

The map $\flat$ in Lemma \ref{puntinisullei} identifies the metric measure space $X$ we started with the 
metric boundary of a graph $G$. From now on, we simply write $\partial G=X$. 
The theorem below says that such identification carries all the important metric structure of $X$.

\begin{theorem}
\label{lengthmetric}
\hfill
\begin{itemize}
\item[(i)] The metric $\orho$, restricted to $X$, is bi-Lipschitz equivalent to $\rho$:
$$
c\rho(x,y)\le\orho(x,y)\le c\rho(x,y).
$$\\
\item[(ii)] Let $Z_k=\{z^k_\alpha:\ \alpha\in I_k\}$ be the set of the distinguished points $z^k_\alpha\in Q^k_\alpha$ mentioned in Theorem A; and let $G_k=\{\alpha\in G:\ d(\alpha)=k\}$. Then, $\rho|_{Z_k}$ is bi-Lipschitz equivalent to $\orho|_{G_k}$,
$$
c\rho(z^k_{\alpha},z^k_{\beta})\le\orho(\alpha,\beta)\le c\rho(z^k_{\alpha},z^k_{\beta}).
$$
\end{itemize}
\end{theorem}

The theorem says that all  Ahlfors-regular spaces (bounded and complete, so far, but we think that these 
further assumptions can be easily removed) are - modulo a bi-Lipschitz change of the metric - 
the boundaries of rather regular and concrete graphs.  Or, in other terms, that all such metric spaces 
admit a natural extension to a metric measure space in which the distance is given by length of curves 
and it is realized, it can be shown, by geodesics. The theme of geodesics is discussed in 
depth in\cite{Geo}.

By fattening the graph's edges, we might even see that the space $X=\partial G$ is the metric boundary 
(modulo bi-Lipschitz change of metric) of a Riemannian manifold $M$ which maintains many of the properties 
of $X$. $M$, for instance, could be taken to be Ahlfors regular, having as topological and Hausdorff 
dimension the least integer $n>Q$. 

\proof 
(i) Let $x,y$ be points in $X$. We show first that $\orho(x,y)\le c\rho(x,y)$. Suppose that $\rho(x,y)\le\delta^k$ and let $\alpha,\beta\in I_{k}$ be such that $x\in \overline{Q^{k}_\alpha},\ y\in \overline{Q^{k}_\beta}$. Then, there is an edge of $G$ between $\alpha$ and $\beta$, hence $\orho(\alpha,\beta)\le\delta^k$. Arguing like in the proof of $\flat$'s surjectivity, we can find sequences $a=\{\alpha_n\}$ and $b=\{\beta_n\}$ such that $\alpha_0=\alpha$, $\beta_0=\beta$; $\beta_n,\alpha_n\in I_{k+n}$;  $x\in \overline{Q_{\alpha_n}}$ and $y\in \overline{Q_{\beta_n}}$. Two successive elements of the sequence $a$ have graph distance at most $2$, and the same holds for $b$. A geometric series applied to the sequences shows then that $\orho(\alpha,x)\le c\delta^k$ and $\orho(y,\beta)\le c\delta^k$. Then, $\orho(x,y)\le c\delta^k$, as wished.

In the other direction, assume that $\rho(x,y)>0$ (otherwise there is nothing to prove), with $x=\flat([\{\alpha_n\}])$ and $y=\flat([\{\beta_n\}])$. By Lemma \ref{puntinisullei}, for each $\epsilon>0$ we can take $n=n(\epsilon)$ large enough to have: 

(i) $\orho(x,\alpha_n),\orho(y,\beta_n)\le\epsilon\rho(x,y)$ (because $\{\alpha_n\}$ and $\{\beta_n\}$ are Cauchy sequences tending to $x$ and $y$, respectively, in $\overline{G}$);

(ii) $\rho(x,\alpha_n),\rho(y,\beta_n)\le\epsilon\rho(x,y)$ (by definition of $\flat$); 

(iii) $n\ge n_0$ to be chosen. 

Let $[\gamma_0=\alpha_n,\gamma_1,\dots,\gamma_m=\beta_n]$ be any path in $G$ which joins $\alpha_n$ and $\beta_n$, and choose distinguished points $w_j\in Q_{\gamma_j}$ ($o\le j\le m$). We have:

\begin{eqnarray*}
\rho(x,y)&\le&\rho(x,w_0)+\sum_{j=1}^m\rho(w_j,w_{j-1})+\rho(w_m,y)\crcr
&\le&2\epsilon\rho(x,y)+\mbox{diam}(Q_{\alpha_n})+\mbox{diam}(Q_{\beta_n})\crcr
&+&\sum_{j=1}^m\rho(w_j,w_{j-1})+\rho(w_m,y)\crcr
&\ &\mbox{by\ (ii)}\crcr
&\le&4\epsilon\rho(x,y)+c\ell_G([\gamma_o,\dots,\gamma_m]),
\end{eqnarray*}

by (iii) and by definition of $\ell_G$. Passing to the infimum of the lengths,

\begin{eqnarray*}
(1-4\epsilon)\rho(x,y)&\le&c\orho(\alpha_n,\beta_n)\crcr
&\le& c[\orho(x,y)+\orho(x,\alpha_n)+\orho(y,\beta_n)]\crcr
&\le&c\orho(x,y)+2c\epsilon\rho(x,y),
\end{eqnarray*}

by (i). It suffices to choose $\epsilon$ small enough to get the desired inequality.

Part (ii) of the theorem is proved similarly.

\endproof

A consequence of the proof is that, given $x,y\in X$,  there exists a ``near geodesic''  in $\overline{G}$ for the distance $\orho(x,y)$, passing through a point at level $k$, where $\rho(x,y)\approx\delta^k$. This fact, indeed, extends the well known relations between Euclidean and hyperbolic distances, and geodesics, in the upper half space.

When the space $(X,\rho)$ has a metric defined by a length, like the spaces introduced by Heinonen and Koskela in \cite{HeK}, then we find ``near geodesics'' joining $x,y\in X$ at all levels $n$ of $\overline{G}$ such that $\delta^n\le \rho(x,y)$.

When the space is poor of rectifiable curves, as it happens with ``snowflake'' metrics on the real line, the near geodesics joining $x$ and $y$ cannot be completely contained in a small strip of $\overline{G}$: they have to reach a level $n$ such that $\delta^n\gtrsim\rho(x,y)$.

\subsection{The Tree and its Boundary as Metric Spaces}

\subsubsection{The Boundary of a Tree.}

Let $T$ be a tree with root $o$ (not necessarily the tree arising from Christ's Theorem). We assume that $T$ has 
\it bounded degree \rm and that \it it has no childless vertices\rm: there is $N\ge1$ such that each vertex $\alpha$ in $T$ has $N(\alpha)$ children, with $1\le N(\alpha)\le N$.

We introduce some notation which will be frequently used. If $\alpha$ is a vertex of $T$, the \it predecessor set \rm of $\alpha$ is $\PPP(\alpha)=[o,\alpha]$, while the \it successor set \rm of $\alpha$ is $\SSS(\alpha)=\{\beta\in T:\ \alpha\in\PPP(\beta)\}$. Given $\alpha$ and $\beta$ in $T$, $\alpha\wedge\beta=\max\left(\PPP(\alpha)\cap\PPP(\beta)\right)$ is their \it confluent. \rm The \it predecessor \rm $\alpha^{-1}$ of $\alpha\in T\setminus\{o\}$ is the only element $\beta$ in $\PPP(\alpha)$ such that $d(\alpha,\beta)=1$.

The \it combinatorial boundary \rm $\tilde{\partial}T$ of $T$ is the set of all half-infinite geodesics, 
with respect to the edge counting distance, having an endpoint at the root.  
For ease of notation, we think of $\tilde{\partial}T$ as a set of labels $\zeta$ 
of the geodesics $\Gamma_\zeta$.  The topology for $\tilde{\partial}T$ is that having as basis the sets
$$
\partial\SSS(\alpha)=\{\zeta\in \tilde{\partial}T:\ \alpha\in\Gamma_\zeta\}.
$$
Let now $\delta\in(0,1)$ be fixed. Associate to each edge $(x^{-1},x)$ of $T$ ($x^{-1}$, we recall, 
is the parent of $x$) the weight  $w(x,x^{-1})=\delta^{d(x,o)}$, where $d$ is the edge-counting distance 
in $T$.

The metric $\rho_T$ on $T$ with parameter $\delta$ is defined as the length-metric associated with the 
weight $w$. The geodesics for $\rho_T$ are clearly the same as the geodesics for the edge-counting metric.
Actually, the poverty of the tree structure is such that all length-metrics have the same geodesics; 
and the edge-counting metric is a particular length metric. Let $\oT$ be the metric completion of $T$ 
with respect to $\rho_T$. The \it metric boundary \rm $\partial T$ of $T$ is $\oT\setminus T$, with the 
topology induced by the metric $\rho_T$.

\begin{lemma}
\label{obvious}
For each $\zeta$ in $\tilde{\partial}T$, consider the geodesic $\Gamma_\zeta$ as a sequence of points in $T$. Then, 

\begin{itemize}
\item[(i)] $\Gamma_\zeta$ is a Cauchy sequence for $\rho_T$;\\
\item[(ii)] Let $\natural(\zeta)$ be the equivalence class of $\Gamma_\zeta$ in $\partial T$. The map $\zeta\mapsto\natural(\zeta)$ is a homeomorphism of $\tilde{\partial}T$ onto $\partial T$.
\end{itemize}
\end{lemma}
The proof of the lemma is easy and it is left to the reader.

Extend $\natural$ to a map from $T\cup\tilde{\partial}T$ to $T\cup\partial T$ by letting $\natural|_T=Id|_T$. Requiring the extended $\natural$ to be a homeomorphism gives a topological structure to $\oT=T\cup\tilde{\partial}T$.

We extend the tree notation to take into account the boundary.  If $\zeta\in\partial T$, let 
$\PPP(\zeta)=\Gamma_\zeta$ is the geodesic from the root to $\zeta$. 

For $\alpha,\beta\in\oT$, let
$$
\alpha\wedge\beta=\max(\PPP(\alpha)\cap\PPP(\beta))
$$
be the confluent of $\alpha$ and $\beta$. Observe that the geodesic $[\alpha,\beta]$ between 
$\alpha$ e $\beta$ in $\oT$ with respect to the metric $\rho_T$ is
$$
[\alpha,\beta]=[\PPP(\alpha)\cup\PPP(\beta)]\setminus\PPP((\alpha\wedge\beta)^{-1}),
$$
thinking of geodesics as sets of vertices, rather than sequences of adjacent edges.

If $\alpha$ is an element of $T$, $\partial\SSS(\alpha)\subseteq\partial T$ is the set of the half-infinite geodesics passing through $\alpha$: $\partial\SSS(\alpha)=\overline{\SSS(\alpha)}\setminus\SSS(\alpha)$.  The metric $\rho_T$ can be explicitly computed. If $\zeta,\xi\in\oT$, then

\begin{equation}
 \label{rhot}
\rho_T(\zeta,\xi)=
\frac{2\delta}{1-\delta}
\left[\delta^{d(\zeta\wedge\xi)}-\frac{1}{2}\left(\delta^{d(\zeta)}+\delta^{d(\xi)}\right)\right].
\end{equation}
In some references $\rho_T$ is called the \it Gromov metric \rm on $\oT$. 
The restriction of $\rho_T$ to $\partial T$ is
$$
\rho_T(\zeta,\xi)=\frac{2\delta}{1-\delta}\delta^{d(\zeta\wedge\xi)},
$$
which is called \it ultrametric \rm in another stream of the vast literature concerning metric spaces.

The balls in $(\partial T,\rho_T)$ are exactly the cl-open sets $\partial\SSS(\alpha)$ and $T$ itself might be seen as the tree of the metric balls in $\partial T$.

\subsubsection{The tree $T$ associated with the graph $G$.}
Let now $T$, the rooted tree, and $\delta>0$ be the same as in Christ's Theorem. The metric $\rho_T$ just defined on $\oT$ is the length-metric associated with the length $\ell_T$, which is the restriction of the length $\ell_G$ to the edges of $T$.  The trivial estimate $\rho_T(\alpha,\beta)\ge\orho(\alpha,\beta)$ cannot in general be reversed.

We can define a Borel measure $\tilde{ m}$ on $\partial T$ by declaring that
$$
\tilde{ m}(\partial\SSS(\alpha)):= m(Q_\alpha).
$$

It is easy to see that the definition is consistent (by Christ's Theorem).
\begin{proposition}
\label{bazzano}
The space $(\partial T,\rho_T,\tilde{ m})$ is complete and Ahlfors $Q$-regular.
\end{proposition}

The statement follows immediately from the fact noted above, that the metric balls in $\partial T$ are the sets of the form $\partial\SSS(\alpha)$.

\subsection{The Map \texorpdfstring{$\Lambda:\partial T\to X$}{Lambda from the boundary of the tree to X}}

Let $P_{T}(\xi)=(\xi_n)$ be the geodesic starting at the root $o$ and ending at the boundary point $\xi\in\partial T$ (we might, and sometimes will, identify $\xi\equiv P_{T}(\xi)$).  Each $\xi_n\in T$ appearing in $P_T(\xi)$ can be identified with a dyadic box $Q(\xi_n)$ in $I_n$. 

We define a map $\Lambda:\partial T\to X$,
$$
\Lambda:\ \xi\mapsto\Lambda(x)=\cap_{n\ge0}\overline{Q(\xi_n)}.
$$

By the discussion above, $\Lambda$ is a contraction,
$$
\orho(\Lambda(\zeta),\Lambda(\xi))\le\rho_T(\zeta,\xi).
$$
The map $\Lambda$ is not open, however, neither is it one-to-one, in general. 

\begin{lemma}
\label{onetoone}
There is a constant $c$ independent of $x$ such that
$$
\sharp\{\zeta\in\oT:\ \Lambda(\zeta)=x\}\le c.
$$
\end{lemma}

\proof Let $\zeta^{(1)},\dots,\zeta^{(n)}$ be points in $\partial T$ such that  $\Lambda(\zeta^{(j)})=x$, $\zeta^{(j)}=(\zeta^{(j)}_k)_{k\in\NN}$.  Let $N=\max\{d(\zeta^{(j)}\wedge\zeta^{(i)}):\ i\ne j=1,\dots,n\}$. Then,$x\in\overline{\zeta^{(1)}_{N+1}}\cap\dots \cap\overline{\zeta^{(n)}_{N+1}}$, and the dyadic sets $\overline{\zeta^{(j)}_{N+1}}$ are all distinct. By Lemma \ref{edue} (i), $n\le c_5$. 
\endproof

\begin{lemma}
\label{onto}
$\Lambda$ maps $\partial T$ onto $X$.
\end{lemma}

\proof 
Let $x$ be a point of $X$, $k\ge0$, consider those $Q^k_\alpha$ such that $x\in\overline{Q^k_\alpha}$, 
and let $J_k$ be their set. 
Each dyadic set in $J_{k+1}$  must have its parent in $J_k$: if $x\in Q^{k+1}_\alpha$, this is obvious; 
if $x\in\partial Q^{k+1}_\alpha$, and $Q^k_\beta$ is the parent of $Q^{k+1}_\alpha$, then $x\in\overline{Q^k_\beta}$. 
This shows that $J=\cup_kJ_k$ is a full subtree of $T$: 
$o\in J$ and $\alpha,\beta\in J\implies[\alpha,\beta]\subseteq J$. 
$J$ has elements at each level $k\ge0$ and it has bounded degree. 
By K\"onig's Infinity Lemma in graph theory, $J$ contains an infinite geodesic $\zeta$, and clearly 
$\Lambda(\zeta)=x$. The last assertion could also be deduced directly from the Axiom of Choice. 
\endproof

At this point, we can be more explicit about the shape of our \it ``near geodesics''. \rm  Let  $x\ne y\in X$ and let $k\ge0$ an integer such that $\delta^{k+1}\le\rho(x,y)\le\delta^k$. Consider  tree geodesics $\Gamma_x,\Gamma_y$ starting at level $k$ and going all the way down to $x$, $y$ respectively, and let $\gamma_k(x)$ and $\gamma_k(y)$ be their points at level $k$. Clearly, $\rho(Q_{\gamma_k(x)},Q_{\gamma_k(y)})\le c\delta^k$. Hence, there is a path $\Gamma^\prime$ in $G$, having length $\ell_G(\Gamma^\prime)\lesssim\delta^k$ joining $\gamma_k(x)$ and $\gamma_k(y)$. Furthermore, there is a fixed constant $h\in\mathbb{N}$ such that the path $\Gamma^\prime$ consists of at most $h$ edges between level $k$ and level $k-h$. The path $\Gamma_x\cup\Gamma^\prime\cup\Gamma_y$ is clearly a near geodesic between $x$ and $y$.

\subsection{Moving Measures through \texorpdfstring{$\Lambda$}{Lambda}}

We can use the map $\Lambda$ to push Borel measures from $\partial T$ to $X$. If $\nu$ is a measure on 
$\partial T$, let
$$
\Lambda_*\nu(E):=\nu(\Lambda^{-1}(E)),
$$
whenever $E$ is Borel measurable in $X$. We want to move measures in the other direction, as well. Let $\omega$ be a nonnegative, Borel measure on $X$, 
and -for each Borel subset $A$ of $\partial T$ and $x\in X$, set
$$
N(x)=\sharp\{\zeta\in\partial T:\ \Lambda(\zeta)=x\},\ N_A(x)=\sharp\{\zeta\in A:\ \Lambda(\zeta)=x\},
$$
and define
\begin{equation}\label{pulback}
 \Lambda^*\omega(A):=\int_X\frac{N_A(x)}{N(x)}d\omega(x).
\end{equation}

Note that, contrary to the pushforward $\Lambda_*$, the operator $\Lambda^*$ is non canonical.   The integral (\ref{pulback}) is well defined because $x\mapsto N_A(x)$ is Borel measurable, a fact which will be proved below.

For each $x$ in $X$, let $\nu_x(A)=\frac{\sharp(\Lambda^{-1}(x)\cap A)}{\sharp(\Lambda^{-1}(x))}$ be the normalized counting measure on $\Lambda^{-1}(x)$. Then
$$
\int_{\partial T}\varphi(\zeta)d\Lambda^*\omega(\zeta)=\int_X\left\{\int_{\Lambda^{-1}(x)}\varphi(\zeta)d\nu_x(\zeta)\right\}d\omega(x)
$$
whenever $\omega$ is a Borel measure on $X$ and $\varphi$ is Borel measurable.

We need the following lemma, whose proof is rather tedious, but reveals the inner functioning of the map $\Lambda$.

\begin{lemma}
 \label{alidosi} 
 Each Borel set $A$ in $\partial T$ can be decomposed $A=\coprod_i A_i$ as the disjoint, countable union of Borel sets $A_i$ such that the restriction $\Lambda|_{A_i}$ of $\Lambda$ to each $A_i$ is $1-1$.  Moreover, if 
$A$ is a Borel set in $\partial T$. Then, $\Lambda(A)$ is a Borel set in $X$.
\end{lemma}

\proof  Let $N$ be the counting function just introduced and define the stopping time
$$
n(x):=\min\{n\in\NN:\ x \ \mbox{belongs\ to\ the\ closure\ of\ }N(x)
\ \mbox{qubes\ at\ level\ }n\}.
$$

We need the sets
$$
H_{k,n}=\{x:\ N(x)=k\ \mbox{and}\ n(x)=n\},
$$
and their pieces
$$
H_{k,n}(\alpha_1,\dots,\alpha_k)=H_{k,n}\cap\overline{Q}_{\alpha_1}\cap\dots\cap\overline{Q}_{\alpha_k},
$$
where $\alpha_1,\dots\alpha_k\in\II_n$ (index set for the qubes at level $n$) are distinct and their order in the labeling 
of the set $H_{k,n}(\alpha_1,\dots,\alpha_k)$ does not matter. 

We \bf claim \rm that each set $H_{k,n}(\alpha_1,\dots,\alpha_k)$ is Borel measurable in $X$.  
Some of these sets and of the sets introduced below are empty, but this causes us no trouble.  
Set $E_k=\{x\in X:\ N(x)\ge k\}$. Then, $E_1=X$ and
$$
E_k=\bigcup_n\bigcup_{\genfrac{}{}{0pt}{}{\alpha_1,\dots,\alpha_k}{\mbox{\small distinct\ in\ }\II_n}}\left(\overline{Q}_{\alpha_1}\cap\dots\cap\overline{Q}_{\alpha_k}\right).
$$
Each $E_k$ is Borel measurable in $X$.  Set now $F_k:=E_k\setminus E_{k+1}=\{x\in X:\ N(x)=k\}$. Set further $G_{1,1}=F_1$, 
and
$$
G_{k,n}=F_k\cap\left[\bigcup_{\genfrac{}{}{0pt}{}{\alpha_1,\dots,\alpha_k}{\mbox{\small distinct\ in\ }\II_n}}
\left(\overline{Q}_{\alpha_1}\cap\dots\cap\overline{Q}_{\alpha_k}\right)
\right].
$$
Then, $H_{k,n}=G_{k,n}\setminus G_{k,n-1}$ and the claimed measurability of the sets $H_{k,n}(\alpha_1,\dots,\alpha_k)$ 
follows.

%For fixed $\alpha\in\II_n$, let then

%\begin{eqnarray*}

%H_{k,n}(\alpha)&=&\bigcup_{\alpha_1=\alpha,\dots,\alpha_1\atop\mbox{\small distinct\ in\ }\II_n}H_{k,n}(\alpha_1,\dots,\alpha_k)\crcr

%&=&H_{k,n}\cap\overline{Q}_\alpha.

%\end{eqnarray*}

Set now, for $1\le j\le k$, $\HHH^j_{k,n}(\alpha_1,\dots,\alpha_k)$ to be the subset of $\partial T$ defined by
$$
\HHH^j_{k,n}(\alpha_1,\dots,\alpha_k)=
\Lambda^{-1}\left(H_{k,n}(\alpha_1,\dots,\alpha_k)\right)\cap\partial S(\alpha_j).
$$

We have the following properties, easy to verify.

\begin{enumerate}
 \item[(i)] $\HHH^j_{k,n}(\alpha_1,\dots,\alpha_k)$ is Borel measurable in $\partial T$;
 \item[(ii)] For each $1\le j\le k$ and each choice of distinct $\alpha_1,\dots,\alpha_k$ in $\II_n$,  $\Lambda$ maps $\HHH^j_{k,n}(\alpha_1,\dots,\alpha_k)$ bijectively onto $H_{k,n}(\alpha_1,\dots,\alpha_k)$;
\item[(iii)] The sets $\HHH^j_{k,n}(\alpha_1,\dots,\alpha_k)$ are mutually disjoint.
\end{enumerate}

We use now the sets just introduced to define a set family in $\partial T$. 
Namely, $\GG$ is the family of the sets having the form
$$
\coprod_{\genfrac{}{}{0pt}{}{k,n;1\le j\le k}{\alpha_1,\dots,\alpha_k\in\II_n}}
\left[\HHH^j_{k,n}(\alpha_1,\dots,\alpha_k)
\cap\Lambda^{-1}(A(k,n;j;\alpha_1,\dots,\alpha_k))\right],
$$
where each set $A(k,n;j;\alpha_1,\dots,\alpha_k)$ is Borel measurable in $X$. Since $\Lambda$ is continuous, $\GG$ is contained in the Borel $\sigma$-algebra of $\partial T$.

The family of sets $\GG$ has in addition the properties listed below.
\begin{enumerate}
 \item[(i)] $\GG$ is a $\sigma$-algebra;
\item[(ii)] The image under $\Lambda$ of each set in $\GG$ is measurable in $X$;
\item[(iii)] Each basic set $\partial S(\alpha)$ belongs to $\GG$.
\end{enumerate}

A consequence of (i),(iii) from the latter property list and (i) from the former property list, 
is that $\GG$ \it is \rm the Borel $\sigma$-algebra in $\partial T$ and that the image of a Borel set of $\partial T$ 
under $\Lambda$ is Borel in $X$. i.e., both $\Lambda$ and $\Lambda^{-1}$ map Borel sets to Borel sets.

Statement (i) is easily verified:
\begin{eqnarray*}
\partial T\setminus \left\{\coprod_{\genfrac{}{}{0pt}{}{k,n;1\le j\le k}{\alpha_1,\dots,\alpha_k\in\II_n}}
\left[\HHH^j_{k,n}(\alpha_1,\dots,\alpha_k)\cap
\Lambda^{-1}(A(k,n;j;\alpha_1,\dots,\alpha_k))\right]\right\}=
\crcr
\coprod_{\genfrac{}{}{0pt}{}{k,n;1\le j\le k}{\alpha_1,\dots,\alpha_k\in\II_n}}
\left[\HHH^j_{k,n}(\alpha_1,\dots,\alpha_k)\cap\Lambda^{-1}
\left(X\setminus A(k,n;j;\alpha_1,\dots,\alpha_k)\right)\right]
\end{eqnarray*}
and
\begin{eqnarray*}
\bigcup_\lambda\left\{\coprod_{\genfrac{}{}{0pt}{}{k,n;1\le j\le k}{\alpha_1,\dots,\alpha_k\in\II_n}}\left[\HHH^j_{k,n}(\alpha_1,\dots,\alpha_k)\cap\Lambda^{-1}(A_\lambda(k,n;j;\alpha_1,\dots,\alpha_k))\right]\right\}=
\crcr
\coprod_{\genfrac{}{}{0pt}{}{k,n;1\le j\le k}{\alpha_1,\dots,\alpha_k\in\II_n}}\left[\HHH^j_{k,n}(\alpha_1,\dots,\alpha_k)\cap\Lambda^{-1}\left(\bigcup_\lambda A(k,n;j;\alpha_1,\dots,\alpha_k)\right)\right]
\end{eqnarray*}

For statement (ii), observe that 

\begin{eqnarray*}
\Lambda\left(\HHH^j_{k,n}(\alpha_1,\dots,\alpha_k)\cap\Lambda^{-1}(A(k,n;j;\alpha_1,\dots,\alpha_k))\right)=
\crcr
H_{k,n}(\alpha_1,\dots,\alpha_k)\cap A(k,n;j;\alpha_1,\dots,\alpha_k),
\end{eqnarray*}

which is measurable in $X$.

We now consider (iii). For fixed $\alpha$ in $T$, we want to show that $\partial S(\alpha)$ is an element of $\mathcal G$. Consider the sets
$$
A=\coprod_{\beta\ge\alpha}\coprod_{k,\ n=d(\beta)}\coprod_{\alpha_1,\dots,\alpha_k:\ \alpha_1=\beta}\HHH^1_{k,n}(\alpha_1,\dots,\alpha_k)
$$
and
$$
B=\coprod_{o\le\gamma\le\alpha}\coprod_{k,n=d(\gamma)}\coprod_{\alpha_1,\dots,\alpha_k:\ \alpha_1=\gamma}\left(\HHH^1_{k,n}(\alpha_1,\dots,\alpha_k)\cap\Lambda^{-1}(\overline{Q_\alpha})\right).
$$

By definition, $A,B\in{\mathcal G}$. It is clear that $A\subseteq\partial S(\alpha)$. We show now that $B\subseteq\partial S(\alpha)$. Let 

$$
\zeta\in \HHH^1_{k,n}(\alpha_1,\dots,\alpha_k)\cap\Lambda^{-1}(\overline{Q_\alpha}),
$$

with $n=d(\gamma)$ and let $x=\Lambda(\zeta)$. Then: $x$ has $k$ preimages at level $n$ and  $\zeta$ is the the only preimage in $\partial S(\gamma)$. Suppose that $\zeta\notin\partial S(\alpha)$. Since $x\in\overline{Q}_\alpha$,there is some other preimage $\zeta^\prime\ne\zeta$ of $x$ which lies in $\partial S(\alpha)$. But $\partial S(\alpha)\subseteq\partial S(\gamma)$, hence, there are two preimages of $x$ in $\partial S(\gamma)$, a contradiction.

We now show that $\partial S(\alpha)\subseteq A\cup B$. Let $\zeta\in\partial S(\alpha)$, $\Lambda(\zeta)=x$, $N(x)=k$ and $n(x)=n$. Consider two cases. 
\begin{itemize}
\item[(a)] First suppose $n(x)\ge d(\alpha)$. Then there are $k$ distinct sets $\overline{Q}_{\alpha_1},\dots,\overline{Q}_{\alpha_k}$ at level $n$ such that $x\in H_{k,n}(\alpha_1,\dots,\alpha_k)$ and $\zeta\in \HHH^j_{k,n}(\alpha_1,\dots,\alpha_k)$ for some $1\le j\le k$. It has to be $\alpha_j\ge\alpha$, because otherwise $\alpha$ and $\alpha_j$ are not order related, hence $\partial S(\alpha)$ and $\partial S(\alpha_j)$ are disjoint and they can not both contain $\zeta$. 

\item[(b)] Suppose $n(x)<d(\alpha)$. Again,  $x\in H_{k,n}(\alpha_1,\dots,\alpha_k)$ and $\zeta\in \HHH^j_{k,n}(\alpha_1,\dots,\alpha_k)$ for some $1\le j\le k$. We want to prove that $\gamma:=\alpha_j<\alpha$. If such is not the case, then $S(\alpha_j)$ and $S(\alpha)$ are not order-related, as above, hence they are disjoint.
\end{itemize}
\endproof

We list the basic properties of $\Lambda^*$.

\begin{proposition}
\label{lambdaomega}
The measure ${\Lambda^*(\omega)}$ is a Borel measure on $\partial T$. Moreover,
\begin{enumerate}
 \item[(i)] $\Lambda_*({\Lambda^*(\omega)})=\omega$;
\item[(ii)] Let $N=\max\{N(x):\ x\in X\}$. Then, for each Borel measurable set $A$ in $\partial T$,
$$
\frac{\omega(\Lambda(A))}{N}\le{\Lambda^*(\omega)}(A)\le\omega(\Lambda(A)).
$$
\end{enumerate} 

About the supports, we have:
\begin{enumerate}
\item[(iii)] $\mbox{supp}(\Lambda^*\omega)\subseteq\Lambda^{-1}(\supp(\omega))$;
\item[(iv)] $\supp(\Lambda_*\nu)=\Lambda(\supp(\nu))$.
\item[(v)] $\supp(\omega)\subseteq\Lambda(\supp(\Lambda^*\omega))$.
\end{enumerate}
\end{proposition}

A consequence of (v) is that $\Lambda^{-1}(\supp(\omega))\subseteq\Lambda^{-1}(\Lambda(\supp(\Lambda^*\omega)))$.  We believe that more is true:  
\begin{conjecture}
$\Lambda^{-1}(\supp(\omega))\subseteq\supp(\Lambda^*\omega)$ (hence, by (iii), $\Lambda^{-1}(\supp(\omega))=\supp(\Lambda^*\omega)$). 
\end{conjecture}
Unfortunately we do not have a proof for this.

\begin{proof}[Proof of Proposition \ref{lambdaomega}]
First, we show that the function $N_A$ is measurable, so that the integral defining ${\Lambda^*(\omega)}$ makes sense. By Lemma \ref{alidosi}, the set $A$ can be decomposed as the disjoint, countable union of measurable subsets $A_i$ of $\partial T$, with the property that $\Lambda$ is $1-1$ on each of them and $\Lambda(A_i)$ is measurable in $X$. Then,
$$
N_A=\sum_i N_{A_i}=\sum_i\chi_{\Lambda(A_i)},
$$
which is Borel measurable.

The set function ${\Lambda^*(\omega)}$ is additive. Let $\{A_n\}_{n=0}^\infty$ be a disjoint family of measurable sets in $\partial T$:
\begin{eqnarray*}
\int_X\frac{N_{\cup{A_n}}(x)}{N(x)}d\omega(x)&=&
\int_X\frac{\sharp\{\zeta\in \cup A_n:\ \Lambda(\zeta)=x\}}{N(x)}d\omega(x)\crcr
&=&\sum_n\int_X\frac{\sharp\{\zeta\in A_n:\ \Lambda(\zeta)=x\}}{N(x)}d\omega(x)\crcr
&=&\sum_n\int_X\frac{N_{{A_n}}(x)}{N(x)}d\omega(x)
\end{eqnarray*}

Estimate (ii) follows from Lemmas \ref{onetoone} and \ref{onto}.  Given a measurable set $A$ in $\partial T$ and $x$ in $\Lambda(A)$: $1\le N_A(x)\le N(x)\le c$. To show (i), let $E$ be a measurable subset of $X$:

\begin{eqnarray*}
\Lambda_*({\Lambda^*(\omega)})(E)&=&{\Lambda^*(\omega)}(\Lambda^{-1}(E))=\int_X\frac{N_{\Lambda^{-1}(E)}(x)}{N(x)}d\omega(x)\crcr
&=&\int_X\frac{\sharp\{\zeta\in \Lambda^{-1}(E):\ \Lambda(\zeta)=x\}}{N(x)}d\omega(x)\crcr
&=&\int_X\chi_E(x)\frac{N(x)}{N(x)}d\omega(x)=\omega(E).
\end{eqnarray*}

We prove (v). Let $K=\supp(\Lambda^*\omega)$. Then,

\begin{eqnarray*}
0&=&\Lambda^*\omega(\partial T\setminus K)=
\int_X\frac{N_{\partial T\setminus K}(x)}{N(x)}d\omega(x)
\end{eqnarray*}
if and only if $\omega-a.e.\ (x)$ we have $0=N_{\partial T\setminus K}(x)=\sharp(\Lambda^{-1}(x)\cap[\partial T\setminus K])$, which is equivalent to having $\Lambda^{-1}(x)\subseteq K$ for $\omega-a.e.\ x$; hence, $x\in\Lambda(K)$ for $\omega-a.e.\ x$. Since $\Lambda(K)$ is closed in $X$, $\supp(\omega)\subseteq\Lambda(K)=\Lambda(\supp(\Lambda^*\omega))$.

We prove (iii). Let $E=\supp(\omega)$. If $x\in E$, $\Lambda^{-1}(x)\subseteq\Lambda^{-1}(E)$, hence 
$\Lambda^{-1}(x)\cap[\partial T\setminus\Lambda^{-1}(E)]=\emptyset$. Thus,
\begin{eqnarray*}
 \Lambda^*\omega(\partial T\setminus\Lambda^{-1}(E))&=&
\int_X
\frac{\sharp(\Lambda^{-1}(x)\cap[\partial T\setminus\Lambda^{-1}(E)])}{\sharp(\Lambda^{-1}(x))}
d\omega(x)\crcr
&=&\int_{X\setminus E}
\frac{\sharp(\Lambda^{-1}(x)\cap[\partial T\setminus\Lambda^{-1}(E)])}{\sharp(\Lambda^{-1}(x))}
d\omega(x)\crcr
&=&0.
\end{eqnarray*}
i.e., $\supp(\Lambda^*\omega)\subseteq\Lambda^{-1}(E)=\Lambda^{-1}(\Lambda^*\omega)$.

Statement (iv) is well known and we include a proof for completeness. For $F$ closed in $\partial T$,
\begin{eqnarray*}
\supp(\Lambda_*\omega)\subseteq F&\iff&\crcr
0&=&\Lambda_*\omega(X\setminus F)=\omega(\Lambda^{-1}(X\setminus F))\crcr
&=&\omega(\partial T\setminus\Lambda^{-1}(F))\crcr
&\iff&\supp(\omega)\subseteq\Lambda^{-1}(F)\crcr
&\iff&\Lambda(\supp(\omega))\subseteq F.
\end{eqnarray*}
\end{proof}

\begin{remark}
\label{noteonmu} 
For the measure $ m$ itself, we have that $N_A(x)/N(x)=1$, $ m$-a.e. 
For each $\alpha\in T$ we have $m(Q_\alpha)=\tilde{m}(\partial S(\alpha))$: 
$ m$, the measure on $X$, and $\tilde{ m}$, the measure on $\partial T$, 
coincide on families of generators of the corresponding Borel $\sigma$-algebras.
Hence, $m(A)=\tilde{m}(\Lambda^{-1}(A))$ for all Borel sets $A$ in $X$ and  
$m(\Lambda(B))=\tilde{m}(B)$ for all Borel sets $B$ in $\partial T$.
We might, and we will,  identify $m$ and $\tilde{m}$.
\end{remark}

We will also need a ``localized'' version of $\Lambda^*$. Let $F$ be a closed subset of 
$\partial T$ and $\omega$ be a Borel measure on $X$. Then,
$$
\Lambda^*_F(A):=\Lambda^*(F\cap A),
$$
is the restriction of $\Lambda^*$ to $F$. Proposition \ref{lambdaomega} localizes to:

\begin{corollary}
\label{lastminute} 
Let $\omega$ be a positive Borel measure on $X$ and $F$ a closed set in $\partial T$. Then,
\begin{enumerate}
\item[(i)] $\supp(\Lambda^*_F\omega)\subseteq\Lambda^{-1}(\supp(\omega))\cap F$.
\item[(ii)] $\Lambda^*_F\omega(A)\approx\omega(\Lambda(F\cap A))$ when $A\subseteq\partial T$
is measurable.
\item[(iii)] $\Lambda_*(\Lambda^*_F\omega)(E)\approx\omega(E\cap\Lambda(F))$ when $E\subseteq X$
is measurable.
\end{enumerate}
\end{corollary}
\begin{proof}
(i) follows from (iii) in Proposition \ref{lambdaomega}. (ii) follows from (ii) in Proposition
\ref{lambdaomega}. (iii) follows form (i) and (ii) in Proposition \ref{lambdaomega}.
\end{proof}

\section{The Muckenhoupt-Wheeden Inequality on Graphs}
\label{sectmw}

\it

In this section we prove the Muckenhoupt-Wheeden inequality on graphs. In the linear case, the inequality can be proved by Fubini's Theorem and easy geometric considerations. 
The main point of the inequality consists in inverting ``on average'' 
the $\ell^1\subseteq\ell^\infty$ inclusion. We need its corollary: an inversion ``on average'' 
of the $\ell^1\subseteq\ell^{p^\prime}$ inclusion, which was independently proved by T. Wolff by means 
of a completely different argument. The other half, which inverts $\ell^{p^\prime}\subseteq\ell^\infty$, 
was independently proved in \cite{AR}. The usefulness of the inequality will be shown in the next section.

\rm

\medskip

Fix $0<s<1$ and $q\ge1$.  For each $x$ in $X$, let $P_G(x)=\{\alpha\in G:\ d_G(\alpha,\Lambda^{-1}(x))\le1\}$, i.e., the set of those $\alpha$ in $G$ having distance at most one from a tree-geodesic ending at $x$.

Given a Borel measure $\omega$ on $X$, define

$$
\II_G\omega(x)=\sum_{\alpha\in P_G(x)}\frac{\omega(\alpha)}{ m(\alpha)^s},
$$
and
$$
\SSS_G\omega(x)=\sup_{\alpha\in P_G(x)}\frac{\omega(\alpha)}{ m(\alpha)^s}.
$$

Here, $\omega(\alpha):=\omega(\overline{Q^k_\alpha})$ and $ m(\alpha):= m(\overline{Q^k_\alpha})= m(Q^k_\alpha)$ are the measures of the corresponding sets. 

Clearly, $\SSS_G\omega\le\II_G\omega$, pointwise. The following surprising theorem of Muckenhoupt and Wheeden shows that, on average, the opposite inequality holds, as well.
\begin{theorem}
 \label{thmb} 
Given $q\ge1$, there is a constant $c_{7}$ such that 
\begin{equation}
\label{mw}
\int_X\II_G\omega(x)^qd m(x)\le c_{7}\int_X\SSS_G\omega(x)^qd m(x).
\end{equation}
\end{theorem}
The inequality was proved in \cite{MW} for Riesz potentials. In \cite{HW}, T. Wolff independently rediscovered, with a different proof, one-half of it. In our context, Wolff's inequality reads:

$$
\int_X\II_G\omega(x)^qd m(x)=\int_X\left(\sum_{\alpha\in P_G(x)}
\frac{\omega(\alpha)}{ m(\alpha)^{s}}\right)^qd m(x)
\le C\int_X\sum_{\alpha\in P_G(x)}\frac{\omega(\alpha)^q}{ m(\alpha)^{sq}}d m(x).
$$

The left hand side is related with a classical capacity on $X$, the right hand side with a new capacity on the tree, 
as we will see below. The equivalence of the two definitions was the key to extend to the nonlinear case important features of linear potential theory \cite{HW}. The other half of (\ref{mw}),

\begin{equation}\label{ar}
\int_X\sum_{\alpha\in P_G(x)}\frac{\omega(\alpha)^q}{ m(\alpha)^{sq}}d m(x)
\le C \int_X\SSS_G\omega(x)^qd m(x),
\end{equation}
was later rediscovered in \cite{AR}, with a completely different proof. 

For ease of the reader, we give the proof of Theorem \ref{thmb} in the present context.  First, define the maximal function
$$
M_G^ m\omega(x)=\sup_{\alpha\in P_G(x)}\frac{\omega(\alpha)}{ m(\alpha)}.
$$
We then have the following well-known lemma whose proof is a standard argument.
\begin{lemma}\label{maximal}
 The maximal function $M_G^ m$ is bounded on $L^\infty$ and on $L^{1,\infty}$,
$$
 m(\{M_G^ m\omega>\lambda\})\lambda\le c\|\omega\|_1
$$
\end{lemma}
\begin{proof}[Proof of Theorem \ref{thmb}]
Again, we separate the proof into some easier steps.

\bf Step 1. \rm Fix $k\in\NN$. Let $\|\omega\|_1=\int_Xd\omega$ and 
using Ahlfors-regularity,

\begin{eqnarray}
\label{stimauno}
\II_G\omega(x)&=&\sum_{\alpha\in P_G(x),\ d(\alpha)< k}\frac{\omega(\alpha)}{ m(\alpha)^s}+
\sum_{\alpha\in P_G(x),\ d(\alpha)\ge k}\frac{\omega(\alpha)}{ m(\alpha)^s}\crcr
&\le&c\|\omega\|_1\delta^{-Qsk}+cM_G^ m\omega(x)\delta^{Q(1-s)k}.
\end{eqnarray}

Choose $k$ in such a way $\delta^{Qk}\approx \|\omega\|_1/M_G^ m\omega(x)$, so that the two summands are approximatively equal,
$$
\|\omega\|_1\delta^{-Qsk}\approx M_G^ m\omega(x)\delta^{Q(1-s)k}.
$$
Then,
\begin{equation}\label{fermo}
 \II_G\omega(x)\le c\|\omega\|_1^{1-s}\left(M_G^ m\omega(x)\right)^s.
\end{equation}

Using Lemma \ref{maximal} and the estimates for $\II_G$, we obtain
\begin{eqnarray}
 m(\II_G\omega>\lambda)&\le& m(c\|\omega\|_1^{1-s}
\left(M_G^ m\omega(x)\right)^s>\lambda)\crcr
&\le&c\|\omega\|_1\lambda^{-\frac{1}{s}}\|\omega\|_1^{\frac{1-s}{s}}
=c\left(\frac{\|\omega\|_1}{\lambda}\right)^{1/s}.
\end{eqnarray}

\bf Step 2. \rm We will show that there exist $a>1$, $b>1$ such that, for all $\epsilon\in(0,1)$ and all $\lambda>0$,

\begin{equation}
\label{goodlambda}
m(\II_G\omega>a\lambda)\le b\epsilon^q m(\II_G\omega>\lambda)+m(\SSS_G\omega>\epsilon\lambda).
\end{equation}

By Corollary \ref{eleffe}, we can replace each of the three sets $A$ whose $ m$-measure is considered in (\ref{goodlambda}) by $A\setminus F$, considering only those points $x$ in $X$ 
such that $\sharp\Lambda^{-1}(x)=1$.

First, we extend the definitions of $\II_G\omega$ and $\SSS_G\omega$ to points $\alpha$ of $T\equiv G$. Let $P_G(\alpha)=\{\beta\in G:\ d_G(\beta,[o,\alpha])\le1\ \mbox{and}\ d(\beta)\le d(\alpha)\}$ and set
$$
\II_G\omega(\alpha):=\sum_{\beta\in P_G(\alpha)}\frac{\omega(\beta)}{ m(\beta)^s},
$$
and
$$
\SSS_G\omega(\alpha):=\sup_{\beta\in P_G(\alpha)}\frac{\omega(\beta)}{ m(\beta)^s}.
$$

If $\II_G\omega(x)>\lambda$ and $x\notin F$, then there is $\alpha\in P_G(x)$ such that $\II_G\omega(\alpha)>\lambda$ and it is clear that, if $\alpha>\beta$ in $T$, then $\II_G\omega(\alpha)\ge\II_G\omega(\beta)$. Let $A(\lambda)$ be the set of the $\alpha$ in $T$ which are maximal with the property $\II_G\omega(\alpha)>\lambda$.

Then, the set $\{\xi\in (X\setminus F)\cup G:\ \II_G\omega(\xi)>\lambda\}$ is the disjoint union of the sets  $S_1(\alpha)=S(\alpha)\cup(\overline{Q^{d(\alpha)}_\alpha}\setminus F)$, as $\alpha$ ranges in $A(\lambda)$.

Observe that $\partial_{\overline{G}} S(\alpha)=\overline{Q^{d(\alpha)}_\alpha}$ is the boundary of $S(\alpha)$ with respect to the metric $\orho$ in $\overline{G}$.

By maximality, if $\alpha$ is in $A(\lambda)$, then $\II_G\omega(\alpha^{-1})\le\lambda<\II_G\omega(\alpha)$.

Fix $\alpha$ in $A(\lambda)$.  We consider two cases. Suppose that for $ m$-almost all $x$ in $\overline{Q^{d(\alpha)}_\alpha}$ we have $\SSS_G\omega(x)>\epsilon\lambda$. Then,
$$
 m(\{\II_G\omega>a\lambda\}\cap \partial_{\overline{G}} S(\alpha))\le m(\partial_{\overline{G}} S(\alpha))
= m(\{\SSS_G\omega>\epsilon\lambda\}\cap \partial_{\overline{G}} S(\alpha)),
$$
and we have nothing to prove.

Suppose that, on the contrary, there is $x_0$ in $\overline{Q_\alpha}\setminus F$ such that $\SSS_G\omega(x_0)\le\epsilon\lambda$. Let 
$$
\tilde{S}(\alpha)=\bigcup_{\genfrac{}{}{0pt}{}{\beta:d(\beta)=d(\alpha)}{d_G(\alpha,\beta)\le1}}\overline{Q_\beta},
$$ 
and let $\omega_1=\omega|_{\tilde{S}(\alpha)}$. 

We write $\partial_\orho\tilde{S}(\alpha)=\cup_{\beta}\partial_\orho S(\beta)$, the union being over the $\beta$'s used in the definition of $\tilde{S}(\alpha)$.By Step 1,

\begin{equation}
\label{dozza}
m(\{\II_G\omega_1>\frac{a\lambda}{2}\}\cap\overline{Q_\alpha})\le\frac{c}{a^{1/s}}\left(\frac{\|\omega_1\|_1}{\lambda}\right)^{1/s}.
\end{equation}

We estimate
\begin{equation}
\label{toscanella}
\|\omega_1\|_1=\omega(\partial_\orho\tilde{S}(\alpha))\le c\SSS_G\omega(\alpha) m(\alpha)^s.
\end{equation}

In fact,
\begin{eqnarray*}
\omega(\partial_\orho\tilde{S}(\alpha))&\le&\sum_{\genfrac{}{}{0pt}{}{\beta:d(\beta)=d(\alpha)}{d_G(\alpha,\beta)\ge1}}\omega(\partial_\orho S(\beta))\crcr
&\le& c\max_{\genfrac{}{}{0pt}{}{\beta:d(\beta)=d(\alpha)}{d_G(\alpha,\beta)\ge1}}\omega(\partial_\orho S(\beta))\crcr
&\ &\mbox{because\ there\ are\ boundedly\ many\ such\ }\beta\mbox{'s}\crcr
&\le&c\left(\max_{\genfrac{}{}{0pt}{}{\beta:d(\beta)=d(\alpha)}{d_G(\alpha,\beta)\ge1}}\frac{\omega(\beta)}{ m(\beta)^s}\right) m(\alpha)^s\crcr
&\le&c\SSS_G\omega(\alpha) m(\alpha)^s.
\end{eqnarray*}

Inserting (\ref{toscanella}) in (\ref{dozza}), 

\begin{eqnarray}
 \label{imola}
 m\left(\{\II_G\omega_1>\frac{a\lambda}{2}\}\cap\alpha\right)&\le&
c\frac{ m(\alpha)}{a^{1/s}}\left(\frac{\SSS_G\omega(\alpha)}{\lambda}\right)^{1/s}\crcr
&\le&\frac{c}{a^{1/s}} m(\alpha)\epsilon^{1/s}\crcr
&\ &\mbox{since\ }\SSS_G\omega(\alpha)\le\SSS_G\omega(x_0)\le\epsilon\lambda,\crcr
&\le&c m(\alpha)\epsilon^{1/s}.
\end{eqnarray}

Consider now $\omega_2:=\omega-\omega_1$. Then,
$$
\II_G\omega_2(\alpha^{-1})\le\II_G\omega(\alpha^{-1})\le\lambda,
$$
by the maximality of $\alpha$. If $x\in\alpha\setminus F$, then,
$$
\II_G\omega_2(x)=\II_G\omega_2(\alpha^{-1})\le\lambda<\frac{a\lambda}{2}.
$$
In the first equality we use the information on $\supp(\omega_2)$. 

Thus,
\begin{eqnarray}
 \label{bubano}
\{\II_G\omega>a\lambda\}\cap(\alpha\setminus F)&=&
\left[\{\II_G\omega_1>\frac{a\lambda}{2}\}\cap(\alpha\setminus F)\right]\cup\left[\{\II_G\omega_2>\frac{a\lambda}{2}\}\cap(\alpha\setminus F)\right]\crcr
&=&\left[\{\II_G\omega_1>\frac{a\lambda}{2}\}\cap(\alpha\setminus F)\right].
\end{eqnarray}
The wished estimate follows from (\ref{imola}) and (\ref{bubano}), proving Step 2.

\smallskip

\bf Step 3. \rm A ``standard argument'' (see \cite{AH}) shows that the estimate follows from the good-$\lambda$ inequality in Step 2.
\end{proof}

For our purposes, the main point of interest of Theorem \ref{thmb} is the following.

\begin{corollary}
\label{mordano}
 $$
\int_X \left(\sum_{\alpha\in P_G(x)}\frac{\omega(\alpha)}{ m(\alpha)^s}\right)^qd m(x)
\le c\int_X\sum_{\alpha\in P_G(x)}\left(\frac{\omega(\alpha)}{ m(\alpha)^s}\right)^qd m(x).
$$
\end{corollary}

\section{Proof of Theorem \ref{maine}}\label{sectcap}

\it

In this section, we first write the energy $\EE(\omega)$ of a measure $\omega$ on $X$ in a way which resembles one side of the Muckenhoupt-Wheeden inequality, the use Wolff's half of the inequality to approximatively write the energy in a different way. The catch is that when the same procedure is applied to a measure on $\partial T$, one gets a very similar expression: the approximating expression for the energy, in fact, does not depend on the \bf graph geometry\it, but only on the \bf tree geometry\it. Using the fact that Borel measures on $\partial T$ and on $X$ can be moved forth and back through $\Lambda$, and observing that energies of corresponding measures are equivalent, we easily obtain a proof of the equivalence of capacities in Theorem \ref{maine}.

\rm

\medskip

\begin{lemma}
 \label{lemmaII}
There are positive constants $c^\prime$, $c^{\prime\prime}$ independent of the Borel measure $\omega$ and 
of the point $x$ in $X$ such that
$$
c^\prime\sum_{\alpha\in P_G(x),\ y\in\alpha} m(\alpha)^{-s}\le K(x,y)\le c^{\prime\prime}\sum_{\alpha\in P_G(x),\ y\in\alpha} m(\alpha)^{-s}
$$
and
$$
c\II_G\omega(x)\le K\omega(x)\le c\II_G\omega(x).
$$
\end{lemma}

\proof
We begin with the estimate from below.
\begin{eqnarray}
\label{usa}
\II_G\omega(x)&=&\sum_{\alpha\in P_G(x)}\frac{\omega(\alpha)}{ m(\alpha)^s}\crcr
&=&\int_{X}\left(\sum_{\alpha\in P_G(x)}\frac{\chi_\alpha(y)}{ m(\alpha)^s}\right)d\omega(y)\crcr
&=&\int_{X}\left(\sum_{\alpha\in P_G(x)\cap P^0_G(y)} m(\alpha)^{-s}\right)d\omega(y),
\end{eqnarray}
where $P^0_G(y)$ is the union of the tree geodesics ending at $y$. Let $\alpha$ be one of the indexes over which the sum is taken. Then, $y\in\alpha$ and there is $\beta\genfrac{}{}{0pt}{}{\sim}{G}\alpha$ such that $x\in\beta$. Then, $\orho(x,y)\le c\delta^{d(\alpha)}$. Also, at each level $k$, there are boundedly many $\alpha$ indexing the sum. Thus,

\begin{eqnarray}\label{canada}
\left(\sum_{\alpha\in P_G(x)\cap P^0_G(y)} m(\alpha)^{-s}\right)&\approx&\sum_{\alpha:\ x\in\alpha,\ \orho(x,y)\le\delta^{d(\alpha)}} m(\alpha)^{-s}\crcr
&\approx& m(x,\orho(x,y))^{-s}\approx K(x,y).
\end{eqnarray}

This is the first estimate in the lemma; the second follows by inserting (\ref{canada}) in (\ref{usa}).
\endproof

Given the kernel $K$ on $X$, we define a new kernel $K_{\partial T}$ on $\partial T$, with $\tilde{ m}$ instead of $ m$ and with $\rho_T$ instead of $\rho\approx\orho$.

\begin{theorem}
\label{treeandgraph}
Given a measure $\omega$ in $X$, let ${\Lambda^*(\omega)}$ be the measure on $\partial T$ defined by (\ref{pulback}), 
and let $\EE_{\partial T}$ be the energy defined by the kernel $K_{\partial T}$. Then,
\begin{equation}
\label{paris}
\EE_X(\omega)\approx\EE_{\partial T}({\Lambda^*\omega}).
\end{equation}

On the other hand, if $\nu$ is a measure on $\partial T$, then

\begin{equation}
\label{toulouse}
\EE_{\partial T}(\nu)\approx\EE_X(\Lambda_*\nu).
\end{equation}
\end{theorem}

\proof

We first prove (\ref{paris})

\begin{eqnarray}\label{lauretta}
 \EE_X(\omega)&:=&\int_XK\omega^\pp d m\crcr
&\approx&\int_X[\II_G\omega]^\pp d m \ \mbox{by\ Lemma\ \ref{lemmaII}}\crcr
&\approx&\int_X\sum_{\alpha\in P_G(x)}
\left(\frac{\omega(\alpha)}{ m(\alpha)^s}\right)^\pp d m(x)\mbox{\ by\ Theorem\ B}\crcr
&=&\sum_{\alpha\in T}\frac{\omega(\alpha)^\pp}{ m(\alpha)^{s\pp}}
\int_X\chi(x\in X:\ \alpha\in P_G(x))d m(x)\crcr
&\approx&\sum_{\alpha\in T}\frac{\omega(\alpha)^\pp}{ m(\alpha)^{s\pp-1}}\crcr
&\approx&\sum_{\alpha\in T}
\frac{{\Lambda^*(\omega)}(\alpha)^\pp}{\tilde{ m}(\alpha)^{s\pp-1}},
\end{eqnarray}
by Proposition \ref{lambdaomega}.

The key observation is that the last expression in the chain of equivalences does not depend on the graph 
structure, but only on the structure of $T$. In fact, particularizing the calculations of (\ref{lauretta})
to the case $X=\partial T$, $G=T$, and to the measure $\mu=\Lambda^*(\omega)$ instead of $\omega$,
we obtain
\begin{equation}\label{angelicaribelle}
 \EE_{\partial T}(\Lambda^*(\omega))\approx \sum_{\alpha\in T}\frac{{\Lambda^*(\omega)}(\alpha)^\pp}{\tilde{ m}(\alpha)^{s\pp-1}}.
\end{equation}
We have then the desired estimate $\EE_{\partial T}({\Lambda^*(\omega)})\approx \EE_X(\omega)$.

We now consider (\ref{toulouse}). Recall the identification $\alpha\equiv\overline{Q_\alpha}$ between points in $G$ and the closure of dyadic sets in $X$.  In the tree $T$, $\alpha$ is identified with a dyadic subregion $[\alpha]\subseteq\partial T$ in the same fashion, and  $[\alpha]\subset\Lambda^{-1}(\alpha)$. The inclusion is generally proper because of the ``edge effect'': elements in the boundary of $\overline{Q_\alpha}$ have preimages in dyadic subsets on $\partial T$ at the same level of $[\alpha]$, but different from $[\alpha]$. 

However, it is still true that $\Lambda^{-1}(\alpha)\subset\cup_{\beta \genfrac{}{}{0pt}{}{\sim}{G}\alpha}[\beta]$, and this, together with Ahlfors-regularity and the uniform bound on the number of such $\beta$'s, suffices to show that (using the identification of the measures $ m$ on $X$ with the measure $\Lambda^* m$ on $\partial T$, see in Remark \ref{noteonmu}):

\begin{eqnarray*}
\EE_{\partial T}(\nu)&\approx&\sum_\alpha\frac{\nu([\alpha])^\pp}{ m(\alpha)^{s\pp-1}}\crcr
&\approx&\sum_\beta\frac{\Lambda_*\nu(\beta)^\pp}{ m(\beta)^{s\pp-1}}\crcr
&\approx&\EE_X(\Lambda_*\nu).
\end{eqnarray*}

\endproof

We are now ready for the proof of Theorem \ref{maine}.

%\begin{theorem}\label{botswana}

%Let $E$ be a compact subset of $X$. Then,

%\begin{equation}\label{france}

%\cpc(E)\approx\cpc_{\partial T}(\Lambda^{-1}(E)).

%\end{equation}

%\end{theorem}

\proof[Proof of Theorem \ref{maine}]
 Observe that the first assertion (\ref{mickey}) is implied 
by the second, since (\ref{mouse}) implies that
$$
\mbox{cap}_{\partial T}(\Lambda^{-1})(E)\approx\mbox{cap}_X(\Lambda(\Lambda^{-1})(E))=\mbox{cap}_X(E).
$$
Let $F\subseteq\partial T$ be closed. We first show that
$$
\mbox{cap}_X(\Lambda(F))\lesssim\mbox{cap}_{\partial T}(F).
$$ 
Let $\omega$ be a positive Borel measure on $X$, supported on $\Lambda(F)$. Then, $\mbox{supp}(\Lambda^*_F\omega)\subseteq F$. Also,

\begin{eqnarray*}
\|\Lambda^*_F\omega\|_1&=&\Lambda^*\omega(F)\crcr
&=&\int_{X}\frac{\sharp(\Lambda^{-1}(x)\cap F)}{\sharp(\Lambda^{-1}(x))}d\omega(x)\crcr
&\approx&\omega(\Lambda(F))=\|\omega\|_1.
\end{eqnarray*}

About energies, since we restrict the sets to $F$, the estimate in (\ref{paris}) is one-side only:

\begin{eqnarray*}
\EE_X(\omega)&\approx&\sum_{\alpha\in T}
\frac{\omega(\overline{Q_\alpha})^\pp}{ m(Q_\alpha)^{\pp s-1}}\crcr
&\approx&\sum_{\alpha\in T}
\frac{\Lambda^*\omega(\partial S(\alpha))^\pp}{ m(Q_\alpha)^{\pp s-1}}\crcr
&\gtrsim&\sum_{\alpha\in T}
\frac{\Lambda^*_F\omega(\partial S(\alpha))^\pp}{ m(Q_\alpha)^{\pp s-1}}\crcr
&\approx&\EE_{\partial T}(\Lambda^*_F\omega).
\end{eqnarray*}

These three information together give:

\begin{eqnarray*}
\mbox{cap}_X(\Lambda(F))&=&
\sup\left\{\frac{\|\omega\|_1^p}{\EE_{X}(\omega)^{p-1}}:
\ \mbox{supp}(\omega)\subseteq\Lambda(F)\right\}\crcr
&\lesssim&
\sup\left\{\frac{\|\Lambda^*_F(\omega)\|_1^p}{\EE_{\partial T}(\Lambda^*_F\omega)^{p-1}}:
\ \mbox{supp}(\Lambda^*_F(\omega))\subseteq F\right\}\crcr
&\le&\mbox{cap}_{\partial T}(F).
\end{eqnarray*}

In the other direction, we show that, if $F$ is closed in $\partial T$, then
$$
\mbox{cap}_X(\Lambda(F))\gtrsim\mbox{cap}_{\partial T}(F).
$$ 
Let $\nu$ be a positive Borel measure supported on $F$. Then, 
$\mbox{supp}(\Lambda_*\nu)=\Lambda(\mbox{supp}(\nu))\subseteq\Lambda(F)$, by Lemma \ref{lambdaomega} (iv). 
Hence,
\begin{eqnarray*}
\mbox{cap}_{\partial T}(F)&=&
\sup\left\{\frac{\|\nu\|_1^p}{\EE_{\partial T}(\nu)^{p-1}}:
\ \mbox{supp}(\nu)\subseteq F\right\}\crcr
&\lesssim&\sup\left\{\frac{\|\Lambda_*\nu\|_1^p}{\EE_{X}(\Lambda_*\nu)^{p-1}}:
\ \mbox{supp}(\Lambda_*\nu)\subseteq \Lambda(F)\right\}\crcr
&\le&\cpc_X(\Lambda(F)).
\end{eqnarray*}
\endproof

\section{Potential Theory on Trees}\label{pttII}

\it

In this section, we develop the basic potential theory on trees.  Most material is covered, with different degrees of generality, in various sources (\cite{So} and \cite{LyPe}, for instance). 

We give a self-contained exposition, only requiring the basic facts of nonlinear potential theory, 
for which we refer to \cite{AH} (Sect. 2.3-2.5). 
The reader will find these facts translated into tree language in the 
Appendix. We list the main features of the section. 
\begin{itemize}
\item[(i)] In \ref{pttone}, we identify the capacity $\Cpc_{\partial T}$ on $\partial T$ considered before with (the restriction to $\partial T$ of) a new tree capacity $\Cpc$ defined on $\oT$. This is analogous to the identification, for subsets of the unit circle, of logarithmic capacity in the complex plane with the linear $1/2$-Bessel capacity on circle. In the Ahlfors-regular case, as a consequence, we have the identification of $\Cpc_X$ with the new tree capacity $\Cpc$.
\item[(ii)] In \ref{pttoneandahalf}, we estimate the capacity $\Cpc$ of points and sets having the form $\partial S(\alpha)$ ($\alpha\in T$). As a corollary, we have estimates for the capacity of metric balls in Ahlfors-regular spaces. It is here explained why the cases $s\in[1/\pp,1)$ are the only interesting ones in the Ahlfors-regular case and why $s=1/\pp$ is especially interesting.
\item[(iii)] We prove in Section \ref{pttthree} a recursive formula to compute capacities of closed subsets of  $\overline{T}$. Indeed, the formula gives two-sided estimates for capacities in Ahlfors-regular metric spaces.
\item[(iv)] We give a simple proof, in Section \ref{ptttwo}, of the \bf Trace Theorem \it (or Carleson Measure Theorem) on trees. The measures satisfying the Trace Inequality (or Carleson imbedding, in the language of most complex analysts) will be characterized by a testing condition, in the spirit of \cite{KS}.

 By the general machinery developed earlier, the Trace Theorem on trees implies an analogous theorem in Ahlfors-regular spaces.

\item[(v)] In \ref{pttfour}, we show that the capacity of a set in the tree can be defined by means of the ``Carleson measures'' supported on this set.
\item[(vi)] In \ref{pttfive}, we give a direct proof -using (v) and a monotonicity property- that the testing condition of the Trace Theorem is equivalent to a certain capacitary condition. An \bf indirect proof \it of this fact is that the capacitary condition, as well, characterizes the Carleson measures.  
\end{itemize}

\rm

\subsection{Tree Capacity Seen from ``Inside'' the Tree.}
\label{pttone}

Let $T$ be a tree, having (combinatorial) boundary $\partial T$ with respect to a fixed root $o$ in $T$, and let $\oT=T\cup\partial T$ be the compactification of $T$. The tree $T$ is endowed with a positive weight $\pi$, $\pi(x)>0\ \forall x\in T$. We consider exponents $1<p<\infty$ and $p^{-1}+\pp^{-1}=1$.  

Let $\xi$ in $\oT$. The predecessor set of $\xi$ is $P(\xi)=\{y:\ o\le y\le \xi\}\cap T$. 
The successor set in $\oT$ of $x\in T$ is $S(x)=\{\xi\in\oT:\ \xi\ge x\}$.  
The natural distance in $T$ is denoted by $d$. Given $x$ and $x$ in $T$, we define
\begin{equation}\label{salonicco}
 d_\pi(x,y)=\sum_{z\in[x,y]}\pi(z)^{1-\pp},
\end{equation}
where $[x,y]$ is the unique geodesic for the distance $d$, here considered as a set of vertices, 
joining $x$ and $y$. The function $d_\pi$ might be considered as the length-distance associated 
with the weight $\pi^{1-\pp}$, but for the fact that $d_\pi(x,x)=\pi(x)^{1-\pp}\ne0$. If we needed
the full force of a distance, we might define the weight $\pi$ on the set of the edges, rather than  
on the set of the vertices; this is what is generally done when studying electrical networks \cite{So}.
Weights defined on vertices, however, perfectly serve our present purposes.

The \it Hardy operator \rm $I$ is defined as
\begin{equation}
\label{hardy}
If(\xi)=\sum_{y\in P(\xi)}f(y),\ \xi\in\oT. 
\end{equation}

Its formal adjoint is
$$
I^*\omega(x)=\int_{S(x)}d\omega(y), x\in T.
$$
where $\omega$ is a Borel measure on $\oT$.

We consider the kernel $g:\oT\times T\to\RR$ defined by
$$
g(\zeta,\alpha)=\chi_{\PPP(\zeta)}(\alpha)=\chi_{\SSS(\alpha)}(\zeta).
$$
Here, we think of $(T,\pi^{1-\pp})$ as a measure space. The set $\oT$ is given the length-metric structure 
$\pi_T:\oT\times\oT\to[0,\infty)$
which assigns to each
edge $(\alpha,\alpha^{-1})$ the length $2^{-d(\alpha,o)}$. Its restriction to the boundary $\partial T$
of the tree $T$ is $\pi_T(\zeta,\xi)=C\cdot2^{-d(\zeta\wedge\xi)}$ 
(the metric itself will play here no r\^ole: we are only interested in the topology).

It is well explained in \cite{AH}
how to develop the potential theory associated with the kernel $g$. 
 See also the Appendix to the present article.
The \it potential \rm of a 
function $f=\pi^{\pp-1}\varphi:T\to\RR$ is $Gf=I(f\pi^{1-\pp})=I\varphi:\oT\to\RR$; while the \it (dual) potential \rm of a positive, Borel measure $\omega$ on $\oT$ is given by $\check{G} \omega=I^* \omega=\int_{S(y)}d \omega$. 

The \it energy \rm of the measure $\omega$ is
$$
\EE(\omega)=\sum_{\alpha\in T}(I^*\omega (\alpha))^{\pp}\pi(\alpha)^{1-\pp}.
$$

We can finally define the \it capacity \rm $\cpc(E)$ of a closed subset $E$ of $\oT$ as
\begin{eqnarray}\label{capacissimo}
\cpc(E)&=&\inf\left\{\|\varphi\|^p_{\ell^p(\pi)}:\ \varphi\ge0,\ I\varphi\ge1\ \mbox{on}
\ E\right\}\crcr
&=&\sup\left\{\frac{\omega(E)^p}{\EE(\omega)^{p-1}}:\ \supp(\omega)\subseteq E\right\}.
\end{eqnarray}

The second equality is a deep result in potential theory (see the appendices), whose proof relies on a min/max principle.

In the important case when $T$ is the tree of Christ's dyadic boxes of an Ahlfors $Q$-regular metric space $(X,\rho, m)$ (which without essential loss of generality, in view of Theorem \ref{maine} and Proposition 
\ref{bazzano}, might be taken to be $\partial T$ itself), and $0<s<1$ is fixed, we consider the weight
$$
\pi_s(\alpha)= m(\overline{Q}_\alpha)^{\frac{s\pp-1}{\pp-1}}= m(\alpha)^{\frac{s\pp-1}{\pp-1}}.
$$ 

The case $s\pp=1$ correspond to unweighted potential theory on trees, which, as we shall see, leads to ``logarithmic'' potentials.

As before, we have identified the measure $ m$ on $X$ with the corresponding measure $\Lambda^* m$ on $\partial T$.

\begin{theorem}
\label{saragozza}
Let $T$ be the tree associated with the metric measure space $(X,\rho, m)$, $0<s<1$ fixed and $\pi=\pi_s$ be the weight just defined. Let $\cpc$ the capacity on $\oT$ associated with $\pi_s$.  Then, there are constants $C_1<C_2$ such that, for all compact subsets $K$ of $\partial T$,
$$
C_1\cpc(K)\le\cpc_X(\Lambda(K))\le C_2\cpc(K).
$$
and
$$
\cpc_X(F)\approx\cpc(\Lambda^{-1}(F))
$$
whenever $F$ is closed in $\partial T$.
\end{theorem}

Theorem \ref{saragozza} is actually implicit in the proof of (\ref{paris}): comparable energies give 
comparable capacities. We have decide to state it separately in this section because the relation 
between $\cpc$ and $\cpc_{\partial T}$, especially in the important case $s\pp=1$, is similar to the 
relation between logarithmic capacity (in the complex plane) and Bessel $1/2$-capacity (on the real line) 
for subsets of the real line.

To better contextualize the capacities in Theorem \ref{saragozza}, consider that, when $\pi=\pi_s$,

\begin{eqnarray*}
\EE(\omega)&:=&\sum_{\alpha\in T}(I^*\omega (\alpha))^{\pp}\pi_s(\alpha)^{1-\pp}\crcr
&=&\sum_{\alpha\in T}\frac{I^*\omega (\alpha))^{\pp}}{ m(\alpha)^{s\pp-1}}\crcr
&=&\int_{\partial T}I\left[(I^*\omega)^{\pp-1}\pi_s^{1-\pp}\right](\zeta)d\omega(\zeta)\crcr
&=&\int_{\partial T}\sum_{\alpha\in T}\frac{I^*\omega^\pp(\alpha)}{ m(\alpha)^{s\pp}}d m(\zeta).
\end{eqnarray*}

Both equalities follow from straightforward applications of Fubini's Theorem.

The integrand on the third line,
\begin{equation}
\label{wolffpotential}
W(\omega)=I\left[(I^*\omega)^{\pp-1}\pi^{1-\pp}\right],
\end{equation}
is the \it Wolff potential \rm of the measure $\omega$; which was introduced by Wolff in \cite{HW}. 
The integral in the fourth line is the intermediate term in Theorem \ref{thmb}, with $q=\pp$. 
A deep analysis of the Wolff potential and of its applications can be found in \cite{COV} and \cite{KV}.

For the remaining part of this section, we will consider general weights $\pi$ on a generic tree $T$: 
proofs are not more difficult and notation is actually more transparent. We will point out, in some cases 
of interest, how the results specialize to the case when $T$ is the tree of the dyadic boxes introduced by 
Christ.

\subsection{Capacity of Special Sets.}
\label{pttoneandahalf}

We begin with a general fact: points having positive capacity are points having finite ``potential theoretic distance'' from the root.

\begin{proposition}
\label{heavypoint}
Let $\zeta_0$ be a point of $\partial T$. Then,
$$
\cpc(\{\zeta_0\})>0\ \iff\ d_\pi(o,\zeta_0)<\infty.
$$
More precisely, $\cpc(\{\zeta_0\})=d_\pi(0,\zeta_0)^{1-p}$ whenever $\zeta_0\in\oT$.
\end{proposition}
\begin{proof}
It suffices to compute the energy of a unit mass $\delta_{\zeta_0}$ concentrated at $\zeta_0$:
\begin{eqnarray*}
\EE(\delta_{\zeta_0})&=&\sum_{\alpha}(I^*\delta_{\zeta_0})^\pp(\alpha)\pi^{1-\pp}(\alpha)\crcr
&=&\sum_{\alpha}\pi^{1-\pp}=d_\pi(\alpha).
\end{eqnarray*}
\end{proof}

It is obvious from the definition of capacity in $\oT$ that $\cpc(\{\alpha\})\ge\cpc({\SSS(\alpha)})$ (if $I\varphi$ is one on $\alpha$, it can be taken to be one on $\partial{\SSS(\alpha)}$ without increasing the norm of $\varphi$). Interesting  cases are those in which the opposite inequality holds as well; for instance, the Ahlfors-regular case.

\begin{lemma}
\label{lastoftheday}
If $0<s<1$ and $\pi=\pi_s$, then $\cpc(\{\alpha\})\approx\cpc(\partial{\SSS(\alpha)})$.
\end{lemma}
\begin{proof}
We have to prove that $\cpc(\{\alpha\})\lesssim\cpc(\partial{\SSS(\alpha)})$. A simple calculation shows that
$$
\cpc(\{\alpha\})=d_{\pi_s}(\alpha)\approx
\begin{cases}
d(\alpha)^{1-p}\ \mbox{if}\ s=1/\pp,\crcr
m(\alpha)^{p(s-1/\pp)}\ \mbox{if}\ 1/\pp<s<1.
\end{cases}
$$

We test the definition of capacity (\ref{capacissimo}) on the measure $\omega:=m|_{\partial S(\alpha)}$, so that $\omega(\partial S(\alpha))=m(\alpha)$.

We estimate the energy of $\omega$:
\begin{eqnarray*}
\EE(\omega)&=&\sum_{\beta\in[o,\alpha)}\frac{m(\alpha)^\pp}{m(\beta)^{s\pp-1}}
+\sum_{\gamma\in S(\alpha)}m(\gamma)^{\pp+1-s\pp}\crcr
&=&I+II.
\end{eqnarray*}

For the first term we have (recognizing a geometric series, by Ahlfors-regularity, when $1/\pp<s<1$):
$$
I\approx
\begin{cases}
d(\alpha)m(\alpha)^{\pp}\ \mbox{if}\ s=1/\pp,\crcr
m(\alpha)^{1+\pp-s\pp}\ \mbox{if}\ 1/\pp<s<1.
\end{cases}
$$

For the second term, using Ahlfors-regularity and the fact that $\pp(1-s)>0$,
\begin{eqnarray*}
II&\approx&
\sum_{n=0}^\infty\left(\sum_{\gamma\in S(\alpha),\ d(\gamma,\alpha)=n}m(\gamma)\right)
\delta^{Q(d(\alpha)+n)\pp(1-s)}\crcr
&=&m(\alpha)\sum_{n=0}^\infty\delta^{Q(d(\alpha)+n)\pp(1-s)}\crcr
&\approx&m(\alpha)^{1+\pp-s\pp}.
\end{eqnarray*}

Summing $I$ and $II$ and using the definition of capacity,

\begin{eqnarray*}
\cpc(\partial S(\alpha))&\ge&\frac{\omega(\alpha)^p}{\EE(\omega)^{p-1}}\crcr
&\approx&\begin{cases}
          \frac{m(\alpha)^p}{[m(\alpha)^\pp d(\alpha)]^{p-1}}\ \mbox{if}\ s=1/\pp,\crcr
\frac{m(\alpha)^p}{[m(\alpha)^{(1-s\pp+\pp)}]^{p-1}}\ \mbox{if}\ 1/\pp<s<1,
         \end{cases}\crcr
&=&\begin{cases}
d(\alpha)^{1-p}\ \mbox{if}\ s=1/\pp,\crcr
m(\alpha)^{p(s-1/\pp)}\ \mbox{if}\ 1/\pp<s<1.
\end{cases}\crcr
&=&\cpc(\{\alpha\}),
\end{eqnarray*}
as wished.
\end{proof}

%Let $E\subseteq T$ be a set and let 

%$\partial{\SSS}(E)=\cup_{\alpha\in E}\partial{\SSS(\alpha)}.$ 

%\begin{corollary}\label{lastlast}

% If $\pi$ is such that $\cpc(\{\alpha\})\approx

%\cpc({\SSS(\alpha)}),$ then

%$$

%\cpc(E)\approx\cpc(\partial{\SSS}(E))

%$$

%for all subsets $E$ of $T$.

%\end{corollary}

By Theorem \ref{maine}, these conclusions apply to the case of the trees coming from Ahlfors-regular spaces and $0<s<1$.  A simple geometric series argument, Proposition \ref{heavypoint} and Lemma \ref{lastoftheday} imply:

\begin{corollary}
\label{heavypointsahlfors}
Let $(X,m,\rho)$ be a (bounded) Ahlfors $Q$-regular space and $0<s<1$. Then,
\begin{enumerate}
 \item[(i)] For all points $x$ in $X$, $\cpc_X(\{x\})=0$ if and only if $\frac{1}{p^\prime}\le s<1$; 
 \item[(ii)] More generally, 
$$
\cpc_X(B(x,r))\approx
\begin{cases}
\log\frac{1}{r}\ \mbox{if}\ s=1/p^\prime,\crcr
r^{Q(sp^\prime-1)}\ \mbox{if}\ 1/p^\prime\le s<1.
\end{cases}
$$
\end{enumerate}
\end{corollary}

As earlier anticipated, the value $s=1/\pp$ relates the capacity $\cpc_X$ with a logarithmic capacity in the tree. Ultimately, this fact is hidden in the Wolff inequality.

\subsection{An Algorithm to Compute Tree Capacities}
\label{pttthree} 
For the material in this subsection, relevant references are also \cite{So} and \cite{BP}.  First, we show that the capacity of a set can be computed as the ``derivative at infinity'' of a Green function. In this case, the r\^ole of the point at infinity is played by the root.

\begin{theorem}
\label{infinito}
Let $E$ be closed in $\oT$ and let $\varphi$ be the corresponding extremal function. Then,
$$
\Cpc(E)=\varphi^{p-1}(o)\pi(o).
$$
\end{theorem}
The Green function in question is the equilibrium potential $\Phi=I\varphi$.

\begin{proof}
Let $E$ be a closed subset of $\oT$ and $\sigma$, $\varphi=(I^* \sigma)^{\pp-1}\pi^{1-\pp}$ be, respectively, its capacitary measure and capacitary function. Then,
$$
\Cpc(E)= \sigma(E)=I^* \sigma(o)=\pi(o)\varphi(o)^{p-1}.
$$
\end{proof}

Let $d_\pi$ be the distance associated with $\pi$: $d_\pi(x)=\sum_{y\in P(x)}\pi^{1-\pp}(y)$. The capacity of a subset of $T$ can be computed by means of a recursive algorithm. 
\begin{theorem}
\label{esplicito}
Let $Z$ be a set in $\oT$, $x_0$ a point in $T$ and let $x_1,\dots,x_n$ be its children. For each $j$, let $Z_j=Z\cap S(x_j)$. If $\Cpc_{x_j}$ denotes capacity with respect to the root $x_j$, then
$$
\Cpc(Z)=\frac{\sum_j\Cpc_{x_j}(Z_j)}{\left\{1+d_\pi(x)\left(\sum_j\Cpc_{x_j}(Z_j)\right)^{\pp-1}\right\}^{p-1}}.
$$
\end{theorem}

The theorem and its proof also hold if the set of the children is infinite.  We need the obvious lemma:

\begin{lemma}
\label{pharmonic}
Let $\varphi=(I^* \sigma)^{\pp-1}\pi^{1-\pp}$, where $ \sigma\in\MM(\oT)$ is a measure. For $x$ in $T$, let $C(x)$ be the set of the children of $x$ in $T$. Then,
$$
\pi(x)\varphi(x)^{p-1}= \sigma(x)+\sum_{x_j\in C(x)}\pi(x_j)\varphi(x_j)^{p-1}.
$$
\end{lemma}
We will apply the Lemma for $x$ outside the support of $\sigma$.

\begin{proof}[Proof of Theorem \ref{esplicito}]
Let $\varphi$ be the extremal function for $\Cpc(Z)$. We claim that 
\begin{equation}
\label{claimuno}
\varphi_j=\frac{\varphi}{1-I\varphi(x)},
\end{equation}
restricted to $S(x_j)$, is extremal for $\Cpc_{x_j}(Z_j)$.

For $\zeta\in Z$, $\sum_{y=x_j}^\zeta\varphi_j(y)=\frac{I\varphi(\zeta)-I\varphi(x)}{1-I\varphi(x)}\ge1$ q.e. $\zeta$, hence $\varphi_j$ is a candidate to be extremal for $\Cpc_{x_j}(Z_j)$. Suppose there exists another $\psi$ such that $\sum_{y=x_j}^\zeta\psi(y)\ge1$ q.e. $\zeta\in Z_j$ and $\|\psi\|_{L^p(S(x_j),\pi)}<\|\varphi_j\|_{L^p(S(x_j),\pi)}$. Reasoning as in the proof of Theorem \ref{cmcap} below, we can find a function $\varphi^\prime$ such that $\|\varphi^\prime\|_{L^p(\pi)}<\|\varphi\|_{L^p(\pi)}$ and $I\varphi^\prime\ge1$ q.e. on $Z$, contradicting the extremality of $\varphi$. The claim is proved.

Next, we claim that
\begin{equation}
\label{claimdue}
\sum_{y=o}^x\varphi(y)^p\pi(y)=d_\pi(x)\Cpc(Z)^\pp.
\end{equation}

Let $y_0=o,\dots,y_N=x$ be an enumeration of the points in the geodesic $[o,x]$ between $o$ and $x$.  Since $E$ is contained in $S(x)$, $\varphi$ has is supported on $S(x)\cup[o,x]$. 

By Lemma \ref{pharmonic}, 
$$
 \varphi(y_j)^{p-1}\pi(y_j)=\varphi(y_{j-1})^{p-1}\pi(y_{j-1})
$$
and, by iteration,
$$
\varphi(y_j)=\frac{\pi(o)^{\pp-1}}{\pi(y_j)^{\pp-1}}.
$$
Summing,
\begin{eqnarray*}
 \sum_{y=o}^x\pi(y)\varphi(y)^p&=&\sum_{y=o}^x\pi(y)^{1+p(1-\pp)}\pi(o)^{p(\pp-1)}\varphi(o)^p\crcr
&=&d_\pi(x)\pi(o)^\pp\varphi(o)^p=d_\pi(x)\Cpc(Z)^\pp,
\end{eqnarray*}
by Lemma \ref{infinito}. This proves \eqref{claimdue}.  Hence,
\begin{eqnarray*}
 \Cpc(Z)&=&\sum_j\sum_{S(x_j)}\varphi^p\pi+\sum_o^x\varphi^p\pi\crcr
&=&\sum_j\sum_{S(x_j)}\varphi_j^p\pi[1-I\varphi(x)]^p+d_\pi(x)\Cpc(Z)^\pp\crcr
&=&\sum_j\Cpc_{x_j}(Z_j)[1-\Cpc(Z)^{\pp-1}d_\pi(x)]^p+d_\pi(x)\Cpc(Z)^\pp (\mbox{by\ equations\ \eqref{claimuno}\ and\ \eqref{claimdue}}),
\end{eqnarray*}
since $I\varphi(x)=\sum_{y=o}^x\varphi(y)=\pi(o)^{\pp-1}\varphi(o)d_\pi(x)$. Thus,
$$
\Cpc(Z)=\sum_j\Cpc_{x_j}(Z_j)[1-\Cpc(Z)^{\pp-1}d_\pi(x)]^p+d_\pi(x)\Cpc(Z)^\pp.
$$

Rewrite the identity as
$$
\left(\sum_j\Cpc_{x_j}(Z_j)\right)^{\pp-1}=
\Cpc(Z)^{\pp-1}\left[1-\Cpc(Z)^{\pp-1}d_\pi(x)\right]^{(1-p)(\pp-1)}
$$
and solve with respect to $\cpc(Z)$.  The recursive formula in the statement of the theorem is proved.\end{proof}

\subsection{Trace Inequalities on Trees and Ahlfors-regular spaces.}
\label{ptttwo}

\subsubsection{Trace Inequalities on Trees.}
To each positive measure $\mu$ we associate
\begin{equation}
 \label{carnorm}
[\mu]:=\sup_{a\in T}\left\{\frac{I^*[(I^*\mu)^\pp\pi^{1-\pp}](a)}{I^*\mu(a)}\right\}^{p-1}.
\end{equation}

By definition, it is homogeneous of degree $d=1$ with respect to $\mu$ and $d=-1$ with respect to $\pi$. Those measures $\mu$ having $[\mu]<\infty$ are called \it Carleson measures for $(I,\pi,p)$\rm.

The following theorem was proved, on trees, in \cite{ARS1} and, in a slightly more general form, 
in \cite{ARS2}. Since the proof in \cite{ARS2} is short, we will give it here,  for convenience of 
the reader. 
See also \cite{ARS4} for the linear case with special weights: 
the proof of the special case is the same as that of the general one.

Trace inequalities and, more generally, weighted trace inequalities have a long history. We mention the approach via testing condition of \cite{KS} and, long before, that by means of capacitary conditions by Maz'ya \cite{Mazya2}.  
See also \cite{KV} for an approach relating trace inequalities and nonlinear equations.

\begin{theorem}
 \label{carlesonmeasures}
There is a constant $C(p)$ which only depends on $p$ such that
\begin{equation}
\label{carleson}
\int_\oT If^pd \mu\le C(p)[ \mu]\sum_{x\in T}f^p(x)\pi(x).
\end{equation}

In fact, we have the quantitative estimates
$$
[ \mu]\le|||I|||_{(L^p(\pi),L^p( \mu))}^p\le C(p)[ \mu].
$$
\end{theorem}

We say that a measure is \it a Carleson measure \rm for the discrete potential $g$ 
 if it satisfies the imbedding 
(\ref{carleson}).
Recall that $g(\zeta,\alpha)=\chi_{P(\zeta)}(\alpha)$.
 We say that $\mu$ satisfies the \it testing condition \rm if $[\mu]<\infty$. 
The theorem says that Carleson measures are exactly those satisfying the testing condition.

%\footnote{Since the proof is very elementary, I'd like to copy it here. I'd also like to deduce from it the Trace Theorem for Ahlfors regular spaces.}

%\{In the file radialvp.tex there is some more material on this.}

\begin{proof}[Proof of the Theorem \ref{carlesonmeasures}]
Inequality (\ref{carleson}) means that 
$$
I:\ell^p(\pi)\to L^p(\mu)
$$
is bounded and, by duality, this is equivalent to the boundedness, with the same norm 
(which we call $|||\mu|||$ in the course of the proof),
of
$$
I^*_\mu:L^\pp(\mu)\to\ell^\pp(\pi^{1-\pp}),
$$
where, it can be easily checked, 
$$
I^*_\mu g(\alpha)=I^*(gd\mu):=\int_{S(\alpha)}gd\mu.
$$

For the duality, we use the $\ell^2$ inner product to have $[\ell^p(\pi)]^*=\ell^\pp(\pi^{1-\pp})$ 
and the $L^2(\mu)$ inner product to have $[L^p(\mu)]^*=L^\pp(\mu)$.

The boundedness of $I^*_\mu$ is expressed by the inequality
\begin{equation}
\label{gelateria}
\sum_{\alpha\in T}|I^*(gd\mu)(\alpha)|^\pp\pi^{1-\pp}(\alpha)\le|||\mu|||^\pp\int_{\oT}|g|^\pp d\mu,
\end{equation}
with a  finite constant $|||\mu|||$, which  obviously has to be checked on positive $g$'s only. We make the left hand side of (\ref{gelateria}) larger. Introduce the maximal operator
$$
\MMM_\mu g(\zeta):=\sup_{\beta\in\PPP(\zeta)}\frac{I^*(gd\mu)}{I^*\mu},
$$
with $\zeta$ in $\oT$ and $g$ positive and measurable on $\oT$. 

We will show that, in fact,
\begin{equation}
\label{angelica}
\sum_{\alpha\in T}[\MMM_\mu g(\alpha)]^\pp I^*(gd\mu)^{\pp}(\alpha)\pi^{1-\pp}(\alpha)
\le|||\mu|||\int_\oT g^\pp d\mu.
\end{equation}

In fact, (\ref{angelica}) follows from the more general inequality
\begin{equation}
\label{laura}
\int_\oT[\MMM_\mu g]^\pp d\sigma\le C(p)\int_\oT g^\pp \MMM_\mu(d\sigma)d\mu,
\end{equation}
which holds for a positive, Borel measure $\sigma$ on $\oT$ with a constant $C(p)$ which only depends on $p$ in $(1,\infty)$. Inequality (\ref{laura}) is proved in Theorem \ref{siblings} below.

Let us verify that (\ref{laura}) implies (\ref{angelica}). Let $\sigma$ be the discrete measure
$$
\sigma(\alpha)=I^*\mu(\alpha)^\pp\pi^{1-\pp}(\alpha),\ \mbox{if}\ \alpha\in T.
$$

The left hand side of (\ref{laura}) is the left hand side of (\ref{angelica}). On the right hand side, observe that
$$
\sup_{\zeta\in\oT}\MMM_\mu(\sigma)(\zeta)=[\mu],
$$
hence,
$$
\int_\oT g^\pp \MMM_\mu(d\sigma)d\mu\le[\mu]\int_\oT g^\pp d\mu.
$$
The proof that $|||\mu|||\le C(p)[\mu]$ is complete.

To prove the other direction, it suffices to test inequality (\ref{gelateria}) on functions of the form $g=\chi_{S(\alpha)}$, $\alpha$ in $T$.
\end{proof}

Observe that testing (\ref{angelica}) on the same functions $g=\chi_{S(\alpha)}$, one obtains that the stronger (\ref{angelica}) and the weaker (\ref{gelateria}) hold together with comparable constants, or together fail. 

\begin{theorem}
\label{siblings}
Given positive Borel measures $\mu$ and $\sigma$ on $\oT$, let 
$$
\MMM_\mu(\sigma):=
\sup_{\beta\in\PPP(\zeta)}\frac{I^*(d\sigma)}{I^*\mu}.
$$
Then, (\ref{laura}) holds.
\end{theorem}
\begin{proof} We show that $\MMM_\mu$ satisfies a weak $(1,1)$ inequality. Let $\lambda>0$ and let $E=\{\zeta\in\oT:\ \MMM_\mu(g)(\zeta)>\lambda\}$. Then, $E$ is \it disjoint \rm union (by the monotonicity of $\MMM_\mu g$ with respect to to the natural partial order on $T$) of sets $S(\alpha_j)$, $j=1,\dots,n$ in $T$ and
%, by maximality of $\alpha_j$ in $S(\alpha_j)$,
$$
\int_{S(\alpha_j)}gd\mu>\lambda\mu(S(\alpha_j))
$$
(by the minimality of $\alpha_j$ in $S(\alpha_j)$).

Thus, $\mu(S(\alpha_j))>0$ and
\begin{eqnarray*}
\sigma(E)&=&\sum_j\sigma(S(\alpha_j))\crcr
&=&\sum_j\frac{\sigma(S(\alpha_j))}{\mu(S(\alpha_j))}\mu(S(\alpha_j))\crcr
&\le&\frac{1}{\lambda}\sum_j\MMM\sigma(\alpha_j)\int_{S(\alpha_j)}gd\mu\crcr
&\le&\frac{1}{\lambda}\sum_{j}\int_{S(\alpha_j)}g\MMM_\mu\sigma d\mu\crcr
&\le&\frac{1}{\lambda}\int_\oT g\MMM_\mu\sigma d\mu.
\end{eqnarray*}

The $L^\infty$ inequality is obvious. Suppose $g\le C$ $\MMM_\mu(\sigma)-a.e.$. We can assume that $\sigma\ne0$ (otherwise (\ref{laura}) holds trivially), hence, $\MMM_\mu\sigma(\zeta)>0$ for all $\zeta$ in $\oT$.  This implies that $g\le C$ $\mu-a.e.$, so that $\int_{S(\alpha)}gd\mu/\mu(S(\alpha))\le C$ for all $\alpha$ in $T$ (set it equal to $0$ if $\mu(S(\alpha))$). Then, $\MMM_\mu g\le C$ everywhere, hence $\sigma-a.e.$  Marcinkiewicz interpolation gives now the theorem.
\end{proof}

\medskip

\noindent
\subsubsection{Trace Inequalities on Ahlfors-regular spaces.} 
Let $K:X\times X\to[0,\infty]$ be the kernel defined in (\ref{nucleo}), with $s\in[1/\pp,1)$, where $X$ is a Ahlfors $Q$-regular space. We say that a positive Borel measure $\mu$ on $X$ \it satisfies the trace inequality \rm for the space $(X,m,\rho)$, the exponent $p$ and the kernel $K$ if the inequality
\begin{equation}
\label{traceineq}
\int_X\left(\int_XK(x,y)f(y)dm(y)\right)^pd\mu\le C(\mu)\int_X f^pdm
\end{equation}
holds for all positive, Borel $f$.

There is an enormous amount of literature on weighted trace inequalities (see, e.g., \cite{KS}, \cite{KV}, \cite{COV} and the literature quoted there) and we make no claim of originality for the results we are going present below. What we are interested in is the relationship between discrete and continuous trace inequalities, and between different necessary and sufficient conditions.

Let $T$ be the tree associated with $X$ and $K_T$ be the kernel on $\partial T$ which corresponds to the same choice of $p$ and $s$. 

\begin{theorem}
\label{tracciateorema}
The measure $\mu$ satisfies the trace inequality for $K$ if and only if $\Lambda^*\mu$ satisfies the trace inequality for $K_T$.
\end{theorem}
\begin{proof}
For a Borel measure $\omega$ on $X$, let $K\omega(x)=\int_X K(x,y)d\omega(y)$. By duality and symmetry of the kernel $K$, $\mu$ satisfies the trace inequality for $K$ if and only if the inequality

\begin{equation}
\label{ammaniti}
\int_X K(gd\mu)^\pp dm\le C(\mu)\int_X g^\pp d\mu
\end{equation}
hold for all positive Borel functions $g$.

That is,
\begin{eqnarray}
C(\mu)\int_X g^\pp d\mu&\ge&\EE_X(gd\mu)\crcr
&\approx&\EE_{\partial T}(\Lambda^*(gd\mu)).
\end{eqnarray}

For $g:X\to\RR^+$, define $\Lambda^*g:=g\circ\Lambda$.  Then one can easily show that
\begin{equation}
\label{noioso}
\Lambda^*(gd\mu)=(\Lambda^*g)d(\Lambda^*\mu).
\end{equation}

To verify \eqref{noioso}, first check the statement for simple $g$, then use Monotone Convergence Theorem.  By the \eqref{noioso}, 
$$
\EE_{\partial T}(\Lambda^*(gd\mu))=\EE_{\partial T}((\Lambda^*g)(d\Lambda^*\mu))
$$
and, since
$$
\int_Xg^\pp d\mu\approx\int_{\partial T}\Lambda^*(g^\pp d\mu)=\int_{\partial T}(\Lambda^*g)^\pp d\Lambda^*\mu,
$$
the wished inequality (\ref{ammaniti}) is equivalent to the tree inequality
$$
\EE_{\partial T}((\Lambda^*g)(\Lambda^*\mu)))\le C(\mu)\int_{\partial T}(\Lambda^*g)^\pp d\Lambda^*\mu.
$$

Set $G=\Lambda^* g$ and $M=\Lambda^*\mu$. The inequality
$$
\EE_{\partial T}(GdM)\le C(M)\int_{\partial T}G^\pp dM
$$
expresses the fact that $M$ satisfies the trace inequality for $K_{\partial T}$.  Hence, if $\Lambda^*\mu$ satisfies the trace inequality on $\partial T$, $\mu$ does on $X$.

\smallskip

In the other direction, if (\ref{ammaniti}) holds, then $\mu$ satisfies the capacitary inequality 
$\mu(E)\le C(\mu)\Cpc_X(E)$, when $E$ is a closed subset of $X$:  test the inequality dual to 
(\ref{ammaniti}) on those $f$ such that $f\ge1$ on $E$. By the equivalence of continuous and discrete 
capacities, this implies that $\Lambda^*\mu(F)\le C(\mu)\Cpc_{\partial T}(F)$ for closed subsets of 
$\partial T$. It is well known (see e.g. \cite{KS}) that the capacitary condition implies, via duality, 
the testing condition $[\Lambda^*\mu]<\infty$ 
(see also the next subsection) and this, in turn, implies the discrete trace inequality by 
Theorem \ref{carlesonmeasures}.
 \end{proof}

\begin{remark}
In view of the results of Section 5.1, the fact that $M$ satisfies the trace inequality for $K_{\partial T}$ is equivalent to the fact that $M$ satisfies (\ref{carleson}). Hence, we have a dyadic  characterization for the measures satisfying the trace inequalities on Ahlfors-regular spaces.
\end{remark}

\begin{corollary}
\label{ultimaspeme}
A nonnegative Borel measure $\mu$ on $X$ satisfies the trace inequality \eqref{traceineq} if and only if it satisfies the testing condition of Kerman-Sawyer type
\begin{equation}
\label{francaviglia}
\int_B\left[\int_B K(x,y)d\mu(y)\right]^\pp dm(x)\le C(\mu)\mu(B)
\end{equation}
uniformly over the metric balls $B$ in $X$. The constant $C(\mu)$ is comparable with $[\mu]$.
\end{corollary}
\begin{proof}
Clearly, the testing condition (\ref{francaviglia}) is implied by the dual (\ref{ammaniti}) to the trace inequality \eqref{traceineq}. In the other direction, it is easy to see that the testing condition (\ref{francaviglia}) implies the testing condition on trees, $[\Lambda^*\mu]<\infty$, which implies the discrete trace inequality, which implies in turn, by Theorem \ref{tracciateorema}, the trace inequality in $X$.
\end{proof}

\subsection{Defining Capacities via Carleson measures}
\label{pttfour}

\begin{theorem}
\label{cmcap}
Let $E\subseteq\oT$ be compact. Then,
\begin{equation}
\label{dimenticavo}
 \Cpc(E)=\sup_{supp( \mu)\subseteq E}\frac{ \mu(E)}{[ \mu]}.
\end{equation}
\end{theorem}

To prove the theorem, we need another characterization of capacity, which holds for the classical capacity, but which is not included in \cite{AH}. Lacking a reference, we give a proof which works in the present context.

%\footnote{Restyle the proof, based on the notes.}

\begin{proposition}
 $$
\cpc(K)=\inf\{\EE( \mu):\ V( \mu)\ge1\ \mbox{on}\ K\}.
$$
\end{proposition}
\begin{proof}
Suppose first that $K\subseteq T$ and that it is finite. Given a measure $ \mu$, let $\varphi=(I^* \mu)^{\pp-1}\pi^{1-\pp}$. Then, $\EE( \mu)=\|\varphi\|_{L^p(\pi)}^p$ and $V( \mu)=I\varphi$. Let 
$$
{\mathcal C}(K)=\inf\{\EE( \mu):\ V( \mu)\ge1\ \mbox{on}\ K\}=\EE(\overline{ \mu}),
$$
where $\overline{ \mu}$ is the extremal measure (it exists by elementary compactness and finiteness of $K$), which satisfies $V(\overline{ \mu})=1$ on $K$ and it is unique.

Let $\overline{\varphi}=(I^*\overline{ \mu})^{\pp-1}\pi^{1-\pp}$. Then,
$$
\cpc(K)=\inf\{\|\varphi\|_{L^p(\pi)}^p:\ I\varphi\ge1\ \mbox{on}\ K\}
\le{\mathcal C}(K)=\EE(\overline{ \mu})=\|\overline{\varphi}\|_{L^p(\pi)}^p.
$$
Suppose there is $ \mu_0$ such that, with $\varphi_0=(I^* \mu_0)^{\pp-1}\pi^{1-\pp}$,
$$
\cpc(K)=\|\varphi_0\|_{L^p(\pi)}^p=\EE( \mu_0)
$$
and
$$
I\varphi_0=V( \mu_0)\ge1\ \mbox{on}\ K.
$$
\end{proof}

\begin{proof}[Proof of Theorem \ref{cmcap}]

We start with inequality $[\le]$.  For fixed $a\in T$ we denote by $\oT_a=\oS(a)$ the subtree of $\oT$ having root $a$ and we add a subscript $a$ to the corresponding tree objects: $Cap_a$ is the capacity of subsets of $\oT_a$,  $\EE_a$ is the energy in $\oT_a$, $\omega_a$ is the extremal measure in the definition of capacity, and so on. Let $E_a=E\cap\oT_a$. The extremal measure $\omega_a$ and the function 
$\varphi_a=(I_a^*\omega)^{\pp-1}\pi^\onepp$ satisfy

\begin{equation}
 \label{alltogethera}
\Cpc_a(E_a)=\omega_a(E_a)=\EE_a(\omega_a)=\|\varphi_a\|_{L^p(\oT_a,\pi)}^p.
\end{equation}

We \bf claim \rm that $\omega_a$ is a rescaling of the extremal measure for $E$ in $\oT$, $\omega$, restricted to $E_a$:
\begin{equation}
 \label{rescale}
\omega_a=\frac{\omega|_{E_a}}{1-I[(I^*\omega)^{\pp-1}\pi^{1-\pp}](a^{-1})}
=\frac{\omega|_{E_a}}{1-V(\omega)(a^{-1})}.
\end{equation}

In fact, $\omega_a$ minimizes $\EE_a( \mu)$ over all measures $ \mu$ such that $I_a[(I_a^* \mu)^{\pp-1}\pi^{1-\pp}](\xi)=V_a( \mu)(\xi)\ge1$ on $E_a$ (with the possible exception of a set having null-capacity). On the other hand, we claim that $\omega|_{E_a}$ minimizes $\EE_a( \mu)$ among all measures $ \mu$ on $E_a$ such that $V_a( \mu)(\xi)\ge1-V(\omega)(a^{-1})$ $q.e.$ on $E_a$. 

Suppose this is not the case. Then there exists a measure $\nu$ on $E_a$ such that $V_a(\nu)(\xi)\ge1-V( \omega)(a^{-1})$ for $q.e.$ $\xi\in E_a$ and 
$$
\EE_a(\nu)=\sum_{ T_a}(I^*\nu)^\pp\pi^\onepp<\sum_{ T_a}(I^*\omega)^\pp\pi^\onepp=\EE_a(\omega|_{E_a}).
$$

Let $\varphi$ be such that $V(\omega)=I\varphi$ in $\oT$ and $\varphi_1$ so that $V_a(\nu)=I_a\nu$ in $\oT_a$.  Define now a new function $\psi$ on $T$:
$$
\psi(x)=
\begin{cases}
 \varphi_1(x)\ \mbox{if}\ x\in T_a,\crcr
\varphi(x)\ \mbox{if}\ x\in T\setminus T_a.
\end{cases}
$$

We have
$$
I\psi(\xi)\ge 1\ q.e.\ \mbox{on}\ E,
$$
hence $\|\psi\|_{L^p(\pi)}^p\ge Cap(E)$. On the other hand,
$$
\|\psi\|_{L^p(\pi)}^p=\EE_a(\nu)+\left[\EE(\omega)-\EE_a(\omega)\right]<\EE(\omega)=Cap(E),
$$
and we have reached a contradiction.

The measure
$$
\lambda=\frac{\omega|_{E_a}}{1-V(\omega)(a^{-1})},
$$
then, minimizes $\EE_a( \mu)$ over the set of the measures $ \mu$ such that $V_a( \mu)(\xi)\ge1$ for $q.e.\ \xi$ in $E_a$, hence $\lambda=\omega_a$.  The \bf claim \rm is proved.

%\footnote{Here there is an implicit uniqueness statement, which is not really necessary. Fix the detail.}

By the homogeneity of the energy,
\begin{eqnarray*}
 \EE_a(\omega|_{E_a})&=&(1-V(\omega)(a^{-1}))^\pp\EE_a(\omega_a)\crcr
&=&(1-V(\omega)(a^{-1}))^\pp\omega_a(E_a)\crcr
&=&(1-V(\omega)(a^{-1}))^{\pp-1}\omega(E_a).
\end{eqnarray*}

As a consequence,
\begin{equation}
 \label{bologna}
\frac{\sum_{x\ge a}(I^*\omega)^\pp\pi^{1-\pp}}{I^*\omega}=\frac{\EE_a(\omega|_{E_a})}{\omega(E_a)}
=(1-V(\omega)(a^{-1}))^{\pp-1}\le1,
\end{equation}
with equality if and only if $a=o$ (we use the default value $V(\omega)(o^{-1})=0$).

Hence, $[\omega]=1$ and
$$
Cap(E)=\omega(E)=\frac{\omega(E)}{[\omega]}.
$$
We now prove $[\ge]$. By definition of $[\cdot],$ $\EE( \mu)\le[ \mu]^{\pp-1} \mu(E)$ for all measures $ \mu$. Then,
$$
\frac{ \mu(E)}{[ \mu]}\le\frac{ \mu(E)}{\left(\frac{\EE( \mu)}{ \mu(E)}\right)^{p-1}}=
\frac{ \mu(E)^{p}}{\EE( \mu)^{p-1}}\le Cap(E),
$$
as wished.
\end{proof}

The proof above has an interesting consequence.

\begin{corollary}  
If $\omega$ is the extremal measure for $\cpc(E)$, with $E$ closed in $\overline{T}$, then $[\omega]=1$.
\end{corollary}

\subsection{Monotonicity of the Tree Condition}
\label{pttfive} 

First, we give a direct proof that the testing condition for the trace inequalities on trees ($[\mu]<\infty$) is monotone: if $\nu\le\mu$, then $[\nu]\lesssim[\mu]$. 

Then, we use this fact to give a direct proof that the testing condition and the capacitary condition are equivalent. This answers a question which was posed to us by Maz'ya a few years ago (private communication).

It is known (see \cite{KS}) that, by a duality argument, the capacitary condition implies the testing condition. We concentrate, then, on the opposite implication.  For a measure $ \mu$ on $\oT$, let $\sigma_\mu=(I^* \mu)^\pp\pi^{1-\pp}$.

\begin{theorem}
\label{monotono}
Let $ \mu$ be a measure on $\oT$ and let $\lambda$ be a measurable function on $\oT$, $0\le\lambda\le1$. If $I^*\sigma_\mu\le I^* \mu$ on $T$, then $I^*\sigma_{\lambda \mu}\le p\cdot I^*(\lambda \mu)$.
\end{theorem}

\begin{corollary}
\label{mono}
If $\nu\le \mu$ and $ \mu,\ \nu$ are measures on $\oT$, then $[\nu]\le p^{p-1}[\mu]$.
\end{corollary}
\begin{proof}
By rescaling, it suffices to verify the hypothesis at the root.   We use a simple argument based on distribution functions.

Let
$$
{\mathcal M}_\mu \lambda(x)=\max_{o\le y\le x}\frac{I^*(\lambda \mu)(y)}{I^* \mu(y)}
$$
be the discrete maximal function we used in \cite{ARS2}. If necessary, we can extend the definition to $x\in\partial T$ in the obvious way. Then,
\begin{eqnarray*}
I^*\sigma_{\lambda \mu}(o)&=&\sum_{x\in T}\left[\frac{I^*(\lambda \mu)(x)}{I^* \mu(x)}\right]^\pp\left(I^* \mu(x)\right)^\pp\pi^{1-\pp}(x)\crcr
&\le&\sum_{x\in T} \left[{\mathcal M}_\mu \lambda(x)\right]^\pp\sigma_\mu(x)\crcr
&=&2\int_0^1t^{\pp-1}\sigma_\mu(\zeta\in\oT:\ {\mathcal M}_\mu \lambda(\zeta)>t)dt.
\end{eqnarray*}

Now, $\{\zeta\in\oT:\ \mathcal{M}_\mu\lambda(\zeta)>t\}=\bigsqcup_{j}S(x_j)$ is 
the disjoint union of Carleson boxes in $\oT$ (by the definition of the maximal function, we do not need to consider the closure of $S(x_j)$ in $\oT$). Then,
\begin{eqnarray*}
t\sigma_\mu(\zeta\in\oT:\ {\mathcal M}_\mu \lambda(\zeta)>t)&=&\sum_jt\sigma_\mu(S(x_j))\crcr
&\le&\sum_jtI^* \mu(x_j)\crcr
&\le&\sum_jI^*(\lambda \mu)(x_j)\crcr
&\le&I^*(\lambda \mu)(o).
\end{eqnarray*}

Inserting this estimate in the previous one and integrating, we have
$$
I^*\sigma_{\lambda \mu}(o)\le\frac{\pp}{\pp-1}\cdot I^*\sigma_{ \mu}(o)=p I^*\sigma_{ \mu}(o).
$$
\end{proof}

An immediate consequence of Theorem \ref{monotono} is that in (\ref{dimenticavo}) we do need to restrict to measures supported in $E$.

\begin{corollary}
\label{corolla}
Let $E\subseteq\oT$ be compact. Then,
\begin{equation}
\label{dimenticavo2}
 \Cpc(E)\le\sup_{\mu}\frac{ \mu(E)}{[\mu]}\le p^{p-1}\Cpc(E).
\end{equation}
\end{corollary}

\subsection{Trace Inequalities: The Testing Condition implies the Capacitary Condition}

We now give a direct proof that the testing condition $[\mu]<\infty$, equivalent to the test inequality \eqref{carleson}, is equivalent to a capacitary condition (see \eqref{capcond} below). Both condition s are known, in a fairly general context, to characterize the Trace inequality; hence they are \it a priori \rm equivalent.  Here, however, we give \ it direct \rm proof of their equivalence. Better, we show that \eqref{carleson} implies \eqref{capcond}, since the opposite (direct) implication is known
(see \cite{KS2}). 

\begin{theorem}
\label{diretto}
Let $\mu$ be a positive, Borel measure on $\oT$. If $ \mu$ satisfies
\begin{equation}
\label{treecond}
\sup_{x\in T}\frac{I^*\left([I^* \mu]^\pp\pi^{1-\pp}\right)(x)}{I^* \mu(x)}\le C_1( \mu).
\end{equation}
then $\mu$ satisfies, for all closed sets $E$ in $\oT$,
\begin{equation}
\label{capcond}
 \mu(E)\le C_2(\mu)\Cpc(E).
\end{equation}
Moreover, $C_2( \mu)\le p^{p-1}C_1(\mu)$.  Conversely, \eqref{capcond} implies \eqref{treecond}.
\end{theorem}
\begin{proof}
Without loss of generality, suppose that $\mu$ satisfies \eqref{treecond} with $C_1( \mu)=1$.  Then, $ \mu_E:= \mu|_{E}\le \mu$ satisfies $[ \mu_E]\le p^{p-1}$ by Theorem \ref{monotono}. Hence,
\begin{eqnarray*}
\Cpc(E)&=&\sup_{\mbox{supp}(\nu)\subseteq E}\frac{\nu(E)}{[\nu]}\crcr
&\ge&\frac{ \mu(E)}{[ \mu_E]}\crcr
&\ge&p^{1-p} \mu(E).
\end{eqnarray*}
\end{proof}

Conditions of testing and of capacitary type are also both known to characterize the \it Carleson measures \rm for the holomorphic Dirichlet space. Let us recall definitions and results. The Dirichlet space ${\mathcal D}$ contains those functions $f$ which are holomorphic in the unit disc $D$ of the complex plane for which the norm
$$
\|f\|_{\mathcal D}^2=|f(0)|^2+\int_D|f^\prime(z)|^2dxdy
$$
is finite. A measure $\mu$ on $\overline{D}$ is \it Carleson \rm for ${\mathcal D}$ if the imbedding ${\mathcal D}\hookrightarrow L^2(\mu)$ is bounded. 
The problem of the boundary values is indeed of interest, when $\mu(\partial D)\ne0$: 
we direct the interested reader to \cite{ARS2} and \cite{Beu} and to the references therein for a 
discussion of this problem. We just mention that such problem is intimately related with the 
characterization itself of the Carleson measures. 
Carleson measures satisfy a sort of ``holomorphic trace inequality'', 
and it is not surpising that they can be  characterized by both capacitary 
and testing conditions. 
For $z=re^{i\theta}$ in $D$, let $S(z)=\{\rho e^{i\varphi}:\ r\le\rho\le1,\ |\theta-\varphi|\le2\pi(1-r)\}$ be the usual \it Carleson box \rm with vertex $z$ and let $I(z)=\partial S(z)\cap\partial D\subset S(z)$ be the part of its boundary lying on $\partial D$. Stegenga \cite{Ste} proved that
the Carleson measures for ${\mathcal D}$ are exactly those satisfying
\begin{equation}
\label{stegenga}
\mu(\cup_{j=1}^n S(z_j))\le C(\mu)\cpc(\cup_{j=1}^n I(z_j)),
\end{equation}
for all finite subsets $\{z_j\}$ of $D$ such that the arcs $I(z_j)$ are pairwise disjoint. 
Here the capacity is nothing other than the logarithmic capacity which corresponds 
to $s=\frac{1}{2}$ and $p=2$ on the unit circle $\mathbb{T}$.  On the other hand, a finite measure 
$\mu$ is Carleson if and only if it satisfies -for all $a$ in $D$- the testing condition
\begin{equation}\label{arsibero}
\int_{S(a)}[\mu(S(z)\cap S(a))]^2\frac{dxdy}{(1-|z|^2)^2}\le C(\mu)\mu(S(a)).
\end{equation}
This condition can be written in a discrete fashion (see, e.g., \cite{ARS1} and \cite{ARS2}), in the 
spirit of the present paper.
We remind the reader that the first testing condition for Carleson measures was found in \cite{KS} and 
it is different from \eqref{arsibero} (it is what we get from \eqref{arsibero} after applying the 
Muckenhoupt-Wheeden-Wolff inequality.  The left hand side of \eqref{arsibero} is analogous to the left 
hand side of \eqref{ar}, while the left hand side of the testing condition in \cite{KS} is analogous to 
the right hand side of \eqref{ar}).

Now, Stegenga's condition involves a measure in the interior and a capacity for a boundary set. 
In Theorem \ref{diretto}, by contrast, we have that measure and capacity either both live in the interior,
 or both live on the boundary. It is then interesting, we believe, to know whether the capacity of the 
interior set is comparable with that of its ``shadow'' on the boundary. The answer on trees, which might 
be transfered to various metric situations, is positive, under a mild assumption on the weight $\pi$.

\begin{lemma}
\label{scatole}
Suppose that, for all $x$ in $T$, 
\begin{equation}
\label{chescatole}
\Cpc(\partial S(x))\gtrsim d_\pi(x)^{1-p}
\end{equation}
(i.e., that $d_\pi(x)^{1-p}\approx\Cpc(\{x\})\approx\Cpc(\partial S(x))$). Then,
$$
\Cpc(S(E))\lesssim\Cpc(E)
$$
(i.e., $\Cpc(S(E))\lesssim\Cpc(E)$) whenever $E=\cup_j\partial S(x_j)$ is finite union of sets of the form $\partial S(x_j)$.
\end{lemma}
Observe that the converse inequality, $\Cpc(S(E))\ge\Cpc(E)$, follows from trivial comparison.

\begin{proof} Let $\varphi$ be the extremal function for $E$ and let $A=\{x_j:\ I\varphi(x_j^{-1})\ge1/2\}$ and $B=\{x_j:\ I\varphi(x_j^{-1})<1/2\}$. Then,
\begin{eqnarray}
\label{unattimo}
\Cpc(E)&=&\sum_{x\in T}\varphi^p(x)\pi(x)\crcr
&\ge&\sum_{x\in\cup_{a\in A}[o,a]}\varphi^p(x)\pi(x)+\sum_{b\in B}\sum_{x\in S(b)}\varphi^p(x)\pi(x)\crcr
&\ge&2^{-p}\Cpc(A)+\sum_{b\in B}(1-I\varphi(b^{-1}))^p\sum_{x\in S(b)}\left(\frac{\varphi(x)}{1-I\varphi(b^{-1})}\right)^{p}\pi(x).
\end{eqnarray}

Since, when $\zeta$ belongs to $S(b)$ for some $b\in B$,
$$
\sum_{x\in[b,\zeta]}\frac{\varphi(x)}{1-I\varphi(b^{-1})}\ge1,
$$
the function ${\varphi(x)}/{(1-I\varphi(b^{-1}))}$ is admissible for $\Cpc_b(\partial S(b))$, hence, 

\begin{eqnarray*}
\sum_{x\in S(b)}\left(\frac{\varphi(x)}{1-I\varphi(b^{-1})}\right)^{p}\pi(x)&\ge&\Cpc_b(\partial S(b))\crcr
&\ge&\Cpc(\partial S(b))\crcr
&\gtrsim&d_\pi(b)^{1-p},
\end{eqnarray*}
by \eqref{chescatole}. Inserting this last inequality in \eqref{unattimo},
\begin{eqnarray*}
\Cpc(E)&\gtrsim&\left[\Cpc(A)+\sum_{b\in B}d_\pi(b)^{1-p}\right]\crcr
&\gtrsim&\Cpc(A)+\sum_{b\in B}\Cpc(\{b\})\crcr
&\ge&\Cpc(A\cup B)=\Cpc(S(E)).
\end{eqnarray*}
\end{proof}

\begin{corollary}\label{finalissima}
Suppose that condition \eqref{chescatole} holds for the weight $\pi$. Then, for a mesaure $\mu$ on $\overline{T}$, the testing condition \eqref{treecond} is equivalent to the capacitary condition
$$
\mu(S(E))\lesssim\Cpc(E),
$$
to be checked over all $E$ subsets of $\partial T$ having the form $E=\cup_j\partial S(x_j)$ (with the set of the $x_j$'s being finite).
\end{corollary}

\section{Appendix: Potential Theory on Trees}\label{pttI}
In \cite{AH}, Sect. 2.3-2.5, the basics of Nonlinear Potential Theory are developed, based on a kernel
$g:X\times M\to[0,+\infty]$, where:
\begin{itemize} 
\item $(M,\nu)$ is a measure space; 
\item $(X,\delta)$ is a separable, complete, locally compact metric space; 
\item $g(\cdot, y)$ is lower 
semicontinuous on $X$ for each $y\in M$; 
\item  $g(x,\cdot)$ is measurable for each $x\in X$.
\end{itemize}
Moreover, an exponent $p$ is chosen in $(1,\infty)$.

With these data, the \it capacity \rm of a set $E$ in $X$ is defined as
$$
\Cpc(E):=\left\{\|f\|^p_{L^p(\nu)}:\ f\ge0,\ \int_Mg(x,y)f(y)d\nu(y)\ge1\ \text{if}\ x\in E\right\}.
$$
Basic object in the theory are the potentials of a nonnegative, measurable function:
\begin{equation}
 \label{potential}
\GG f(x)=\int_M g(x,y)f(y)d\nu(y),\ x\in X;
\end{equation}
and of a Borel, nonnegative measure:
\begin{equation}
 \label{copotential}
\GGc m(y)=\int_X g(x,y)d m(x),\ y\in M.
\end{equation}

On the tree $T$ having root $o$, we set $X=\oT$, where $\oT$ is endowed with any reasonable metric (e.g. the 
length metric obtained by assigning to each edge $(\alpha,\alpha^{-1})$ the length $2^{-d(\alpha)}$) 
and $M=T$ with the discrete measure $\nu=\pi^{1-\pp}$. Our kernel is
$$
g(\xi,x)=\chi_{P(\xi)}(x)=\chi_{S(x)}(\xi).
$$
We shall also let $\varphi=f\nu$, hence, $\varphi^p\pi=f^p\nu$.

With these definitions and renormalizations, the following property summarizes the main results 
from \cite{AH}, Sect. 2.3-2.5.
\begin{proposition}
\label{elenco}
\hfill
\begin{enumerate}
\item $Gf(\zeta)=I(f\nu)(\zeta)=I\varphi$;
\item $\check{G} m(y)=I^* m(y)=\int_{S(y)}d m$;
\item $\Cpc(E)=\inf\left\{\sum \varphi^p\pi:\ \varphi\ge0,\ I\varphi\ge1\ \mbox{on}\ E\right\}$;
\item Let $E\subseteq\oT$. Then, $\Cpc(E)=0$ if and only if there is $\varphi\in L^p_+(\pi)$ such that 
$E\subseteq\{x:\ I\varphi(x)=+\infty\}$. In particular, $E\subseteq\partial T$;
\item Let $\overline{\Omega}_E$ be the closure of $\Omega_E$ in $L^p(\pi)$. Then,
$$
\overline{\Omega}(E):=\left\{\varphi\in L^p_+(\pi):\ I\varphi(\xi)\ge1\ \mbox{for}\ q.e.\ \xi\in E\right\};
$$
\item Let $E\subset \oT$. Then,
$\exists!\ \varphi^E$ such that $\varphi^E\in L_+^p(\pi)$, $I\varphi^E\ge1$ $q.e.$ on $E$, $\Cpc(E)=\sum \varphi^E(x)^p\pi(x)$;
\item The \it mutual energy \rm of $ m$ and $f$ is
$\EEE( m,f)=\int_{\oT}I(f\nu)d m=\sum_T I^* m\cdot f\cdot\nu=\EE( m,\varphi)$;
\item Let $K\subset\oT$ be compact. Then,
\begin{eqnarray*}
\Cpc(K)^{1/p}&=&\sup\left\{ m(K):\mbox{\ supported\ in\ }K,\ \sum_T[I^* m]^{p^\prime}\pi^{1-p^\prime}\le1\right\};
\end{eqnarray*}
\item There exists an extremal, positive  measure $ m^K$ with support in $K$, such that
$$
\varphi^K=[I^* m^K]^{p^\prime-1}\pi^{1-\pp}.
$$
Moreover,
$$
\Cpc(K)= m^K(K)=\sum_T[I^* m^K]^{p^\prime}\pi^{1-\pp}=\int_{\oT}I[\varphi^K] d m^K; 
$$
\item The \it nonlinear potential \rm of a measure $ m$ is
\begin{eqnarray*}
V( m)(\zeta)&=&I(\pi^{1-p^\prime}[I^* m]^{p^\prime-1})(\zeta).
\end{eqnarray*}
The energy of a measure is 
\begin{eqnarray*}
\EE( m)&=&\int_{\oT}V( m)d m=\sum_{x}(I^* m)^{p^\prime}\pi^{1-p^\prime}=
\int_{\oT}I[(I^* m)^{p^\prime-1}\pi^{1-p^\prime}]d m;
\end{eqnarray*}
\item Let $K$ be compact in $\oT$. Then,
$$
I\varphi^K=V( m^K)\le1\ \mbox{on\ supp}( m^K) 
$$
and 
$$
\Cpc(K)=\max\{ m(K):\  m\ge0\ \mbox{and\ supported\ in\ }K,\ \mbox{such that}\ V( m)\le1\ \mbox{on\ supp}( m)\};
$$
\item Let $E\subset\oT$ is such that $\Cpc(E)<\infty$. Then, there is a measure $ m^E$ such that $\varphi^E$, the capacitary function, is given by $\varphi^E=\left(I^* m^E\right)^{\pp-1}\pi^{1-\pp}$. Moreover,
$$
I\varphi^E\ge1\ q.e.\ \mbox{on}\ E,\ I\varphi^E\le1\ q.e.\ \mbox{on}\ \mbox{supp}( m^E)
$$
and
$$
 m^E(\overline{E})=\sum_T\left(I^* m^E\right)^\pp\pi^{1-\pp}=\int_\oT I\varphi^E d m^E=
\Cpc(E).
$$
\end{enumerate}
The capacity functional has the following basic properties:
\begin{proposition}
\label{basic}
\hfill
\begin{enumerate}
\item $\Cpc(\emptyset)=0$; 
\item If $E\subseteq F$, then $\Cpc(E)\le\Cpc(F)$;
\item If $K_i\searrow K$ are compact sets, then $\Cpc(K_i)\searrow\Cpc(K)$;
\item If $E_i \nearrow E$ are arbitrary sets, then $\Cpc(E_i)\nearrow\Cpc(E)$.
\end{enumerate}
\end{proposition}
A set $E$ in $\oT$ is \it capacitable \rm if 
$$
\Cpc(E)=\sup\left\{\Cpc(K):\ K\subseteq E\ \mbox{is\ compact}\right\}.
$$
The Capacitability Theorem guarantees that
all Suslin sets (hence, all Borel sets) are capacitable.

As a consequence of the above and of homogeneity, we have the following equivalent definitions of capacity.

\begin{eqnarray*}
 \cpc(K)&=&\inf\left\{\|\varphi\|_{L^p(\pi)}^p:\ I\varphi\ge1\ \mbox{on}\ K\right\}\crcr
&=&\inf\left\{\|\varphi\|_{L^p(\pi)}^p:\ I\varphi\ge1\ q.e.\ \mbox{on}\ K\right\}\crcr
&=&\inf\left\{\frac{\|\varphi\|_{L^p(\pi)}^p}{\left(\inf_{x\in K}I\varphi\right)^p}:\ I\varphi\ge1\ q.e.\  \mbox{on}\ K\right\}\crcr
&=&\left(\sup\{ m(K):\ \EE( m)\le1,\ \mbox{supp}( m)\subseteq K\}\right)^p\crcr
&=&\sup_{\mbox{supp}( m)\subseteq K}\frac{ m(K)^p}{\EE( m)^{p-1}}\crcr
&=&\left(\inf\{\EE( m):\  m(K)\ge1,\ \mbox{supp}( m)\subseteq K\}\right)^{1-p}\crcr
&=&\sup\{ m(K):\ V( m)\le1\ \mbox{on}\ \mbox{supp}( m),\ \mbox{supp}( m)\subseteq K\}\crcr
&=&\sup_{ m:\ \mbox{supp}( m)\subseteq K}\frac{ m(K)}{\sup_{\xi\in\mbox{supp}( m)}V( m)(\xi)^{p-1}}\crcr
&=&\left(\inf_{ m:\ \mbox{supp}( m)\subseteq K}\left\{\sup_{\xi\in\mbox{supp}( m)}V( m)(\xi)^{p-1}:\  m(K)\ge1\right\}\right)^{-1}\crcr
\end{eqnarray*}

\end{proposition}

\end{document}